\documentclass[12pt,leqno]{article} 

\usepackage{amsmath,amssymb,amsthm}
\usepackage{tikz}
\usepackage{color}
\usepackage[toc]{appendix}
\usepackage{graphicx}
\usepackage{fancyhdr}
\usepackage{enumitem}
\usepackage{bbm}
\usepackage{parskip}
\usepackage{float}
\usepackage{chngpage}
\usepackage{calc}
\usepackage{bigints}
\usepackage{array}
\usepackage{booktabs}
\usepackage{rotating}
\usepackage{multirow}
\usepackage{adjustbox}
\usepackage{tabularx}
\usepackage{verbatim}
\usepackage{mathtools}
\usepackage{ragged2e}
\usepackage[makeroom]{cancel}
\usepackage{caption}
\usepackage{hyperref}
\usepackage{caption}
\usepackage{subcaption}
\usepackage{appendix}
\usepackage{pgfplots}
\pgfplotsset{compat=1.16}
\usepackage[english]{babel}
\usepackage{hyphenat}
\usepackage[makeindex]{imakeidx}
\usetikzlibrary{datavisualization}
\usetikzlibrary{matrix}
\usetikzlibrary{datavisualization.formats.functions}

\usepackage{geometry} 
\usepackage{graphicx} 
\usepackage{float}

\def\dif{\partial}
\def\al{\alpha}
\def\be{\beta}
\def\ga{\gamma}

\def\op#1{{\rm op}(#1)}
\def\sieq{\overset{\Sigma'}{=}}
\def\sieqb{\overset{\Sigma''}{=}}

\def\lr#1{\langle{#1}\rangle}

\def\xig{\langle{\xi}\rangle_{\!\gamma}}

\def\R{{\mathbb R}}
\def\C{{\mathbb C}}

\def\N{{\mathbb N}}

\def\de{\delta}
\def\La{\Lambda}
\def\la{\lambda}
\def\ep{\epsilon}

\def\bro{\bar\rho}

\def\R{\mathbb{R}}
\def\C{\mathbb{C}}

\def\N{\mathbb{N}}

\def\bro{\bar\rho}
 \def\1{\mathbbm{1}}
\newcommand{\vertiii}[1]{{\left\vert\kern-0.25ex\left\vert\kern-0.25ex\left\vert #1 
\right\vert\kern-0.25ex\right\vert\kern-0.25ex\right\vert}}
\usepackage{mathrsfs}
\usepackage{accents}
\numberwithin{equation}{section}

\newtheorem{definition}{Definition}[section]

\newtheorem{remark}{Remark}[section]

\newtheorem{lem}{Lemma}[section]
\newtheorem{prop}{Proposition}[section]

\newtheorem{cor}{Corollary}[section]

\title{Geometric results for hyperbolic operators with spectral transition of the Hamilton map }
\author{ Enrico Bernardi \thanks{Dipartimento di Scienze Statistiche Paolo Fortunati, Universit\`a di Bologna, Bologna, Italy. \textbf{e-mail}: enrico.bernardi@unibo.it} \and
Tatsuo Nishitani\thanks{Department of Mathematics, Graduate School of Science, Osaka University,
  Toyonaka, Osaka, Japan.  \textbf{e-mail}:nishitani@math.sci.osaka-u.ac.jp}
}
\date{} 

\begin{document}
\maketitle
\bigskip
	
\begin{abstract}
In this paper we study a class of non-effectively hyperbolic operators vanishing of
order $ 2 $ on a manifold, on a sub-region of  which the spectral structure of the
Hamilton map changes type. Suitable normal symplectic coordinates are found
together with an analysis of the Hamilton system associated to the
principal symbol and a  factorization result, preparing the operator for a microlocal energy
estimate, is finally proven. 
\end{abstract}
	
Key words and phrases:~~Cauchy problem, Hamilton map and flow, Transition case, Non-effectively hyperbolic operator . \\
	
AMS 2020 classification:Primary: 35L15; Secondary: 35B30. .
	
\allowdisplaybreaks
		
\section{Introduction}\label{sec:intro}

Two seminal papers in the 1970s (\cite{MR427843}), (\cite{MR492751}) devoted to the analysis of
the Cauchy problem for linear hyperbolic operators with multiple
characteristics have been instrumental in highlighting the fundamental
features that one needs to study when one wants to prove
well-posedness results, or for that matter, propagation of
singularities results, for those kind of equations.
Noting that in  what follows the multiplicities of the characteristics
roots of the operator will
be at most $ 2 $, we recall that
these properties are related to the symplectic
geometry of the manifold of double points $ \Sigma $, to the spectral nature
of the fundamental matrix (or Hamilton map), attached to the operator at those double
points and to the behavior of the lower order terms.

More precisely let
\begin{equation}
  \label{eq:1}
  P(x,D) = -D_{0}^{2} + A_{1}(x,D')D_{0} + A_{2}(x,D')
\end{equation}

be a differential operator of order $ 2 $ in $ D_{0} $ with
coefficients $ A_{j}(x,D') $ that are classical pseudodifferential
operators of order $ j $ on $ \R^{n} $, $
x=(x_{0},x')=(x_{0},x_{1},\ldots,x_{n}) $ and the $ A_{j}(x,D') $
depend smoothly on $ x_{0} $. We have $ P(x_{0},x',\xi_{0},\xi') =
p(x,\xi) + P_{1}(x,\xi) + P_{0}(x,\xi) $ with $ p $ the principal
symbol of $ P(x,D) $ and $ P_{j} $, $ j=1,0 $ the symbols of the lower
order terms. We will assume throughout the paper that the symbol $ p $
is hyperbolic w.r.t the $ (x=0,\xi=e_{0}) $ direction, that
\begin{equation}
\label{eq:2ji}
\text{$p$ vanishes exactly to second order on $\Sigma$.}
\end{equation}
and that

\begin{equation}
\label{eq:rank}
{\rm rank\,}\sigma\big|_{\Sigma}={\rm constant}
\end{equation}
where $\sigma=\sum_{j=0}^nd\xi_j\wedge d x_j$ is the symplectic $2$-form. 

Then it is well known that the fundamental matrix $ F_{p}(\rho) =
dH_{p}(\rho) $ at $ \rho \in \Sigma $ can either have two real
eigenvalues, a case known as effectively hyperbolic, completely solved
and understood, see e.g. (\cite{MR4718636}) for a very recent review and
simplification of the proofs in that case,  or else its
spectrum is purely imaginary, i.e.  ${\rm Spec}F_p(\rho)\subset i\R$,
$ \rho\in\Sigma$. When this occurs it is also known that only two possible sub-cases may happen,
see below Lemma (\ref{lem:seni:a}). 

When the type of either one of these two possible sub-cases is fixed on $ \Sigma $, then essentially our
knowledge is almost complete: a very thorough and exhaustive review of the state of the art
can be found, together with some still open questions, in
(\cite{MR3726883}) to which we also refer for some of the notation used in
the following.
However in (\cite{MR2438425}) it was noted that, beyond the
aforementioned conditions on the spectrum of the fundamental matrix,
another ingredient had to be taken into account when establishing the
well posedness of the Cauchy problem for the class of non-effectively operators, namely the behavior of the
Hamilton flow of the simple bicharacteristic curves with respect to
the manifold of double points. The existence of cases when this flow
may land tangentially onto the double manifold, prevents one from getting
the $ C^{\infty} $ well posedness results one should expect, even with
the lower order terms satisfying the conditions outlined in (\cite{MR492751}).
It is indeed possible to show that one cannot prove  well posedness beyond  a certain
Gevrey threshold, very likely related to the geometry of those curves.

While the picture when the type of spectral behavior of the fundamental
matrix is fixed has been therefore essentially almost fully understood, there have
been a number of more recent results dealing with cases when one allows the nature of the spectrum to vary
along the double manifold, see for instance
(\cite{MR1464121}),(\cite{MR2942985}),(\cite{MR1848348}) and especially
(\cite{MR3185898}),(\cite{MR3362026}),(\cite{MR3784872}) where more
general types of transitions were considered, albeit when $
\Sigma $ has a low codimension.

Consequently the purpose of this paper is to continue the investigation of these
transition situations in the case of a general $ C^{\infty} $ double
manifold of any codimension, under the assumption that at every double point the spectrum of
the fundamental matrix is purely imaginary and the transition occurs
where the cardinality of Jordan
blocks in its elementary decomposition varies from $ 2 $ to $ 4 $
crossing a sub-manifold (see Lemma (\ref{lem:seni:a})). Then we  arrange a set-up where the study of
the principal symbol
will yield information on the behavior of the Hamilton flow and on the
possibility of an elementary factorization conducive to weighted
energy estimates in a suitable operator calculus. In particular we improve the
results in  (\cite{MR2942985}) to get a very general normal form under
this specific type of transition and we arrange a system of coordinates
where the function giving rise to the transition region is exhibited
explicitly.

The paper is organized as follows: in Section (\ref{nf1}) we characterize precisely the type of transition under examination
and in Proposition (\ref{pro:form:1}) we exhibit a set of normal
coordinates, where a function $ \theta $ responsible for the change of
type is found and computed. In Section (\ref{exlem}) we
further analyze this function $ \theta $ with an eye toward its future
use in implementing weighted energy estimates, as well as in investigating the behavior of bicharacteristics of the principal symbol.
In Section (\ref{tanbich}) we prove under a number of symplectically invariant
conditions on the coordinate functions found in Section (\ref{nf1})
that there exists a bicharacteristic of the principal symbol tangent
to the double manifold. This clearly is important in light of the
aforementioned lack of well-posedness in the $ C^{\infty} $ category
when the Hamilton flow is attracted to the double manifold.
Finally in Section (\ref{elfac}) we recover
the general conditions allowing the principal symbol to be factorized
as in \cite{zbMATH03625907}, a preliminary step toward the setting up
of a coherent system of weighted energy estimates.


\section{Normal form}
\label{nf1}
\subsection{Some notations}
\label{ssnot}
Assume that the set $\Sigma=\{p=0, \nabla p=0\}$ of critical points of
$p=0$ is a smooth manifold of codimension $d+1$ and that
(\ref{eq:2ji})and (\ref{eq:rank}) hold.

From (\ref{eq:2ji}) without restrictions, we may assume that for any $\rho\in\Sigma$ there is a neighborhood of $\rho$ where  we can write 
\[
p=-\xi_0^2+\Sigma_{j=1}^d\phi_j^2.
\]
Here $d\phi_j$ ($\phi_0=\xi_0$) are linearly independent at $\rho$ and $\Sigma$ is given by $\phi_j=0$, $0\leq j\leq d$. 
In \cite{MR3726883} we find detailed discussions on the Cauchy problem for ${\rm op}(p)$ under the assumptions \eqref{eq:2ji}, \eqref{eq:rank} and $
{\rm Spec}F_p(\rho)\subset i\R$, $ \rho\in\Sigma$ 
assuming further that {\it there is no spectral transition of the
  Hamilton map $F_p$}, that is

\medskip
\noindent
For any $\rho\in \Sigma$ there is a conic neighborhood $V$ of $\rho$ such that either ${\rm Ker\,}F_p^2\cap{\rm Im\,}F_p^2=\{0\}$ or ${\rm Ker\,}F_p^2\cap{\rm Im\,}F_p^2\neq \{0\}$ holds throughout $V\cap\Sigma$.

\medskip
\noindent
Our main concern in this section is to derive a normal form of $p$ under the assumptions \eqref{eq:2ji} and \eqref{eq:rank} in the presence of a spectral transition of the Hamilton map (spectral transition, for short). Let us denote
\[
W(\rho)={\rm Ker}F_p^2(\rho)\cap{\rm Im}F_p^2(\rho),\quad \rho\in \Sigma.
\]
We sometimes denote $\xi_0=\phi_0$. Since
\[
{\rm dim}T_{\rho}\Sigma+{\rm rank}\,\sigma|_{(T_{\rho}\Sigma)^{\sigma}}={\rm dim}(T_{\rho}\Sigma)^{\sigma}+{\rm rank}\,\sigma|_{T_{\rho}\Sigma}
\]
and ${\rm rank}\,\sigma|_{(T_{\rho}\Sigma)^{\sigma}}={\rm rank}(\{\phi_i,\phi_j\}(\rho))$ we see that \eqref{eq:rank} is equivalent to
\begin{equation}
\label{eq:kaisu}
{\rm rank}\big(\{\phi_i,\phi_j\}\big)={\rm constant}=r\, (=2d+{\rm rank}\,\sigma|_{T_{\rho}\Sigma}-2n)
\end{equation}
where $\big(\{\phi_i,\phi_j\}\big)$ denotes the $(d+1)\times(d+1)$ matrix with $(i,j)$-th entry $\{\phi_{i-1},\phi_{j-1}\}$. 

\begin{lem}
\label{lem:seni:a}Assume  \eqref{eq:rank}. We have $W(\bro)\neq \{0\}$ if the spectral transition occurs at $\bro$.
\end{lem}
\begin{proof}Assuming $W(\bro)=\{0\}$ we show $W(\rho)=\{0\}$ near $\bro$ so that there is no
spectral transition. Thanks to \cite[Theorem 1.4.6]{MR492751} (or \cite[Theorem 21.5.3]{MR781536}) one can choose a symplectic basis such that
\begin{gather*}
\text{either}\;\;\;p_{\bro}=-\xi_0^2+\Sigma_{j=1}^{l_1}\mu_j(x_j^2+\xi_j^2)+\Sigma_{j=l_1+1}^{l_2}\xi_j^2,\quad \mu_j>0,\\
\text{or}\;\; p_{\bro}=e(x_0^2-\xi_0^2)+\Sigma_{j=1}^{l_1-1}\mu_j(x_j^2+\xi_j^2)+\Sigma_{j=l_1+1}^{l_2}\xi_j^2,\quad \mu_j>0,\;\;e>0
\end{gather*}
where $\pm e$, $\pm i\mu_j$ are non-zero eigenvalues of $F_p(\bro)$ and $2l_1=r$ for \eqref{eq:kaisu}. From the continuity of the eigenvalues of $F_p(\rho)$ with respect to $\rho\in\Sigma$, $F_p(\rho)$ has at least $r$ non-zero eigenvalues near $\bro$ (with counting multiplicity). If $W(\rho)\neq \{0\}$ thanks to \cite[Theorem 1.4.6]{MR492751} (or \cite[Theorem 21.5.3]{MR781536}) in a suitable symplectic basis one can write
\[
p_{\rho}=-\xi_0^2+2\xi_0\xi_1+x_1^2+\sum_{j=2}^{l'_1}\mu'_j(x_j^2+\xi_j^2)+\sum_{j=l'_1+1}^{l'_2}\xi_j^2,\quad \mu'_j>0
\]
where $l_1'=l_1$ because of \eqref{eq:kaisu} so that $F_p(\rho)$ has $r-2$ non-zero eigenvalues (with counting multiplicity) which is a contradiction proving $W(\rho)=\{0\}$ near $\bro$.
\end{proof}
\medskip
\noindent
In what follows we always work near $\bro$ but do not mention this and it should be understood that $F_p$ is defined only on $\Sigma$ near $\bro$ .
\subsection{A normal form }
\label{sec:form}

Denote 
\[
\Sigma'=\{\phi_j=0, 1\leq j\leq d\}
\]
and by $O^k(\Sigma')$ a smooth $f(x,\xi')$ vanishing of order $k$ on $\Sigma'$ near $\bro$ and write $O(\Sigma')$ for $O^1(\Sigma')$. We shall write $f_1\sieq f_2$ to mean $f_1-f_2=O(\Sigma')$.
\begin{prop}
\label{pro:form:1}Assume \eqref{eq:2ji}, \eqref{eq:rank} and that the spectral transition occurs at $\bro$ then one can write 
\begin{equation}
\label{eq:form:0}
p=-(\xi_0+\phi_1)(\xi_0-\phi_1)+\theta\phi_1^2+\Sigma_{j=2}^r\phi_j^2+\Sigma_{j=r+1}^d\phi_j^2
\end{equation}
with a smooth $\theta$ $($positively homogeneous of degree $0$$)$ such that $1+\theta> 0$ where
\begin{equation}
\label{eq:form:1}
\{\phi_i,\phi_j\}\sieq 0,\quad 0\leq i\leq d,\;\;j\geq r+1,\quad \{\xi_0-\phi_1,\phi_j\}\sieq 0,\quad 0\leq j\leq d
\end{equation}
and
\begin{equation}
\label{eq:form:2}
\{\phi_1,\phi_2\}(\bro)\neq 0,\; \{\phi_2,\phi_j\}\sieq 0, \; 3\leq j\leq r,\; {\rm det}(\{\phi_i,\phi_j\}(\bro))_{3\leq i,j\leq r}\neq 0.
\end{equation}
If $p$ has the form  \eqref{eq:form:0} with \eqref{eq:form:1} and \eqref{eq:form:2} then $\theta$ is equal to the product of nonzero eigenvalues of $F_p(\rho)$ $($with counting multiplicity$)$ multiplied by a positive smooth function on $\Sigma'$ whenever $W(\rho)=\{0\}$. Moreover 
$\theta(\rho)=0$ if and only if $W(\rho)\neq \{0\}$ and that
\begin{equation}
\label{eq:te:seifu}
\begin{split}
\theta(\rho)>0&\Longleftrightarrow \text{$p$ is non-effectively hyperbolic at $\rho$ with $W(\rho)=\{0\}$},\\
\theta(\rho)=0&\Longleftrightarrow \text{$p$ is non-effectively hyperbolic at $\rho$ with $W(\rho)\neq \{0\}$},\\
\theta(\rho)<0&\Longleftrightarrow  \text{$p$ is effectively hyperbolic at $\rho$.}
\end{split}
\end{equation}
\end{prop}
We say that $p$ is of normal form if $p$ is written as \eqref{eq:form:0} with $\phi_j$ satisfying \eqref{eq:form:1} and \eqref{eq:form:2}. When $p$ is noneffectively hyperbolic at $\rho$ we say noneffectively hyperbolic of type 1 or of type 2 according to $W(\rho)=\{0\}$ or $W(\rho)\neq \{0\}$. This normal form, especially $\theta$ is not unique. To see this we first note that
\begin{lem}
\label{lem:kakikae}One can rewrite
\[
-(\xi_0+\phi_1)(\xi_0-\phi_1)+\theta\phi_1^2=-(\xi_0+\tilde\phi_1)(\xi_0-\tilde\phi_1)+\hat\theta\tilde\phi_1^2
\]
with $
\tilde\phi=(1+\nu)\phi_1$ and $ \hat\theta=(\theta-\nu^2-2\nu)/(1+\nu)^2$.
\end{lem}
\begin{proof} The proof is clear.
\end{proof}
If $\nu\sieq 0$ in Lemma \ref{lem:kakikae}, it is clear that \eqref{eq:form:1} and \eqref{eq:form:2} hold replacing $\phi_1$ by $\tilde \phi_1$ such that
\[
p=-(\xi_0+\tilde\phi_1)(\xi_0-\tilde \phi_1)+\hat\theta\phi_1^2+\Sigma_{j=2}^r\phi_j^2+\Sigma_{j=r+1}^d\phi_j^2
\]
is also a normal form and that $\theta\sieq \hat\theta$. This is no
coincidence.

Let $\Pi\lambda_j(\rho)$ be the product of nonzero eigenvalues of $F_p(\rho)$ (with counting multiplicity). Note that $\Pi\lambda_j(\rho)$ is real because nonzero eigenvalues of $F_p(\rho)$ are $\pm i\mu$ and possibly $\pm \lambda$ with real $\mu, \lambda$. From Proposition \ref{pro:form:1} there is a smooth $f(\rho)$ near $\bro$ with $f(\bro)>0$ such that 
\begin{equation}
\label{eq:to:full}
\theta(\rho)=f(\rho)\Pi \lambda_j(\rho),\quad W(\rho)=\{0\},\;\;\rho\in \Sigma.
\end{equation}
Since $\Pi \lambda_j(\rho)$ is invariant under changes of symplectic coordinates we could say that $\theta|_{\Sigma}$ is conformally invariant. It is also clear from \eqref{eq:to:full} that $\Pi \lambda_j(\rho)$ smoothly extends from the set where $W(\rho)=\{0\}$ to a neighborhood of $\bro$ on $\Sigma$ by setting $\Pi \lambda_j(\rho)=0$ if $W(\rho)\neq0$, the extension is given by $\theta(\rho)/f(\rho)$ through \eqref{eq:to:full}.
\begin{remark}
\label{rem:O:4}\rm
If $
\theta=\tilde\theta+\Sigma_{j=1}^dc_j\phi_j$
choosing $\nu$ such that  $2\nu=\Sigma_{j=1}^dc_j\phi_j$ in Lemma \ref{lem:kakikae} we have
\[
\hat\theta=\tilde\theta+r,\quad r=O^2(\Sigma').
\]
Therefore one can write
\[
-(\xi_0+\phi_1)(\xi_0-\phi_1)+\theta\phi_1^2=-(\xi_0+\tilde\phi_1)(\xi_0-\tilde\phi_1)+(\tilde\theta+O^2(\Sigma'))\tilde\phi_1^2.
\]
That is, if $\tilde\theta|_{\Sigma'}=\theta|_{\Sigma'}$ one can replace $\theta$ by $\tilde\theta+O^2(\Sigma')$  in the normal form.
\end{remark}
This proposition shows that transitions may occur already in very simple examples. For instance, consider
\[
p=-\xi_0^2+(1+\theta(x))\xi_1^2+(x_0+x_1)^2\xi_n^2,\quad |\theta(x)|<1,\;\;\theta(0)=0 
\]
near $(0, e_n)$, $e_n=(0,\ldots,0,1)$ with $n\geq 2$. It is clear that the doubly characteristic set is given by $
\Sigma=\{\xi_0=0,\xi_1=0, x_0+x_1=0\}$. Writing 
\begin{equation}
\label{eq:rei:1}
p=-(\xi_0+\phi_1)(\xi_0-\phi_1)+\theta(x)\phi_1^2+\phi_2^2,\quad \phi_1=\xi_1,\;\;\phi_2=(x_0+x_1)\xi_n
\end{equation}
it is also clear that \eqref{eq:form:1} and \eqref{eq:form:2} are satisfied. By choosing $\theta$ all possible types of transition can occur.

\medskip

\noindent
Proof of Proposition: 
 Denote $A=(\{\phi_i,\phi_j\})_{0\leq i,j\leq d}$ (note that $A$ is
 independent of $\xi_0$) where the rank of $A(\rho)$ is constant
 $r$. Then there exists a smooth basis $v_0, v_{r+1},\ldots, v_d$,
 defined on $\Sigma'$,  of ${\rm Ker}A$. One can assume $v_0={^t}(1,
 v_0')$. To check this it suffices to prove the first entry of some
 $v_0(\bro), v_{r+1}(\bro),\ldots, v_d(\bro)$ is different form
 $0$. Suppose the contrary. Let $ X\in C_{\bro}\cap T_{\bro}\Sigma$
 where $C_{\bro}$ is the propagation cone. One can write
 $X=\Sigma_{j=0}^d\al_jH_{\phi_j}(\bro)$ since $C_{\bro}\subset {\rm
   Im}F_p(\bro)$. From $X\in T_{\bro}\Sigma$ it follows $A\al=0$ hence
 $\al_0=0$ by assumption. Writing $X=(\bar x,\bar\xi)$ we have $\bar
 x_0=0$. Recall $C_{\bro}=\{X; \sigma(X,Y)\leq 0,\;\forall Y\in
 \Gamma_{\bro}\}$. Since for any $(y,\eta')$ there is $\eta_0$ such
 that $(y, \eta)\in \Gamma_{\bro}$ (because
 $\Gamma_{\bro}=\{\eta_0>(\sum_{j=1}^rd\phi_j(y,\eta')^2)^{1/2}\}$) we
 conclude $X=0$. This proves that $C_{\bro}\cap T_{\bro}\Sigma=\{0\}$
 hence $\bro$ is effectively hyperbolic (\cite[Corollary
 1.4.7]{MR492751}, \cite[Lemma 1.1.2]{MR1166190}) so that $W(\bro)=
 \{0\}$ contradicting the assumption.
 Since one can assume $v_0={^t}(1, v_0')$ considering $v_j-(\text{the first component of $v_{j}$})v_0$ one may assume
\[
v_{j}={^t}(0, v_j'(\rho)),\;\;v_j'(\rho)=(v'_{j1},\ldots,v'_{jd}),\;\;\rho\in\Sigma'\quad j=r+1,\ldots,d.
\]
We extend $v'_{r+1},\ldots, v'_d$ outside $\Sigma'$ and orthonormalize them and denote the resulting ones by $w_j'$. 
We choose smooth $w_1',\ldots, w_r'$ such that $w_1',\ldots, w_d'$ will be a smooth orthonormal basis of $\R^d$. Denote 
\[
w_0={^t}(1, 0),\;\;w_j={^t}(0, w_j'),\quad j=1,\ldots,d
\]
then it is clear that $Aw_j\sieq 0$ for $j=r+1,\ldots,d$ and $w_0,
w_1,\ldots,w_d$ is an orthonormal  basis of $\R^{d+1}$. Let $e_j$ be
the unit vector in $\R^{d+1}$ with $j+1$ th component $1$. Denote
$(e_0,\ldots,e_d)=(w_0,\ldots,w_d)P$, $P=(p_{ij})$ so that $P$ is
orthogonal matrix. It is clear that $P=1\oplus P'$.
Denote
\[
{\tilde \phi_i}=\sum_{k=0}^dp_{ik}\phi_k
\]
then we have ${\tilde \phi_0}=\phi_0$ and 
\[
p=-\xi_0^2+\sum_{j=1}^d{\tilde\phi_j^2}.
\]
Noting that $A$ is skew-symmetric we see that $PAP^{-1}\sieq \tilde A\oplus O_{d-r}$ where $O_{d-r}$ is the zero matrix of order $d-r$. Since $(\{{\tilde\phi_i},{\tilde\phi_j}\})\sieq PAP^{-1}$ we have the first assertion of \eqref{eq:form:1} replacing $\tilde\phi_j$ by $\phi_j$. Write $\tilde\phi_j\to \phi_j$ and denote
\[
\tilde A=(\{\phi_i,\phi_j\})_{0\leq i, j\leq r}=\begin{pmatrix}0&{^t}a'\\
-a'&A_1\end{pmatrix},\quad {^t}a'=(\{\xi_0,\phi_1\},\ldots,\{\xi_0,\phi_r\})
\]
where $a'(\bro)\neq 0$ otherwise $\{\xi_0,\phi_j\}(\bro)=0$ for
$j=1,\ldots,d$ so that $F_p(\bro)$ is similar to $F_{\xi_0^2}\oplus
F_{\Sigma_{j=1}^d\phi_j^2}(\bro)$ and hence $W(\bro)=\{0\}$
contradicting with assumption. We now show
\begin{equation}
\label{eq:M:r}
A_1=(\{\phi_i,\phi_j\})_{1\leq i, j\leq r}\quad \text{is non-singular}.
\end{equation}
 Since $P=1\oplus P'$ there is $v={^t}(1,v')$, $v'\neq 0$ such that $\tilde A v=0$, that is $\lr{a', v'}=0$ and $A_1v'=a'$. Suppose that there is $0\neq w'$ such that $A_1w'=0$. Note that $\lr{a', w'}=\lr{A_1v',w'}=-\lr{v', A_1w'}=0$ for $A_1$ is skew-symmetric. This shows that $\tilde A z=0$ with $z={^t}(1,v'+w')$ which contradicts to ${\rm dim\,Ker}\tilde A=1$ (for the rank of $\tilde A$ is $r$).

Define $\al={^t}(\al_1,\ldots,\al_r)\neq 0$ and $\psi$ by
\[
A_1\al=a',\quad \psi=-\sum_{j=1}^r\al_j\phi_j
\]
then since $A_1$ is skew-symmetric it is clear that 
\[
(\{\phi_i,\phi_j\})_{0\leq i,j\leq r}\begin{bmatrix}1\\
\al
\end{bmatrix}=\tilde A\begin{bmatrix}1\\
\al
\end{bmatrix}=0
\]
that is, one has 
\begin{equation}
\label{eq:0:psi}
\{\xi_0-\psi,\phi_j\}\sieq 0,\quad j=0,\ldots, d.
\end{equation}
Choosing a smooth orthogonal matrix $T=(t_{ij})_{1\leq i,j\leq r}$ with the first row $-
\al/|\al|$ 
consider ${\tilde \phi_j}=\sum_{k=1}^rt_{jk}\phi_k$ so that $\sum_{j=1}^r{\tilde\phi_j^2}=\sum_{j=1}^r\phi_j^2$ where 
\[
{\tilde \phi_1}=\psi/|\al|
\]
hence one can write
\begin{equation}
\label{eq:zen:form}
p=-(\xi_0+\tilde\phi_1)(\xi_0-\tilde\phi_1)+\sum_{j=2}^r{\tilde\phi_j^2}+\sum_{j=r+1}^d\phi_j^2.
\end{equation}
Here note that
\[
(\{\tilde\phi_i,\tilde\phi_j\})_{1\leq i,j\leq r}\sieq TA_1T^{-1},\quad \{\xi_0-|\al|\tilde\phi_1, \tilde\phi_j\}\sieq 0,\quad 0\leq j\leq d.
\]
If $\{\tilde\phi_1,\tilde\phi_j\}(\bro)=0$ for $2\leq j\leq r$ then $\{\xi_0,\tilde\phi_j\}(\bro)=0$ for $1\leq j\leq r$ from which we have $W(\bro)= \{0\}$ as before contradicting with the assumption. Thus we may assume $\{\tilde\phi_1,\tilde\phi_2\}(\bro)\neq 0$.
Writing $\psi=|\al|\tilde\phi_1\to \phi_1$ and $\tilde\phi_j\to \phi_j$ ($2\leq j\leq r$) we have 
\begin{equation}
\label{eq:form:1b}
p=-(\xi_0+\phi_1)(\xi_0-\phi_1)+\theta \phi_1^2+\sum_{j=2}^r\phi_j^2+\sum_{j=r+1}^d\phi_j^2
\end{equation}
where
\begin{equation}
\label{eq:te:teigi}
\theta=(1-|\al|^2)/|\al|^2,\quad \{\xi_0-\phi_1,\phi_j\}
\sieq 0,\;\;j=0,\ldots,d,\quad \{\phi_1,\phi_2\}(\bro)\neq 0.
\end{equation}
Thus we have the second assertion of \eqref{eq:form:1} and the first assertion of  \eqref{eq:form:2}.
 If we replace $\tilde\phi_1$ by $|\al|\tilde\phi_1$ then ${\rm det}(\{\tilde\phi_i,\tilde\phi_j\})_{1\leq i,j\leq r}$ is multiplied by $|\al|^2$ hence still 
nonsingular. Therefore we have
\[
{\rm det}(\{\phi_i,\phi_j\})_{1\leq i,j\leq r}\neq 0.
\]
Consider
\[
\tilde\phi_j=\sum_{k=2}^rt_{jk}\phi_k,\quad 2\leq j\leq r
\]
and choosing a suitable (smooth) orthogonal matrix $T=(t_{jk})_{2\leq j, k\leq r}$ with the first row=normalized ${^t}(\{\phi_1,\phi_2\},\ldots,\{\phi_1,\phi_r\})$ such that
\[
\sum_{j=2}^r{\tilde\phi_j^2}= \sum_{j=2}^r\phi_j^2,\quad \{\phi_1,\tilde\phi_j\}\sieq \sum_{k=2}^rt_{jk}\{\phi_1,\phi_k\}\sieq 0,\quad 3\leq j\leq r
\]
one can assume that
\[
\{\phi_1,\tilde\phi_2\}(\bro)=(\Sigma_{j=2}^r\{\phi_1,\phi_j\}^2)^{1/2}\neq 0,\quad \{\phi_1,\tilde\phi_j\}\sieq 0,\;\;3\leq j\leq r.
\]
Since it is clear that $\{\xi_0-\phi_1,\tilde\phi_j\}\sieq 0$, $0\leq j\leq d$ 
writing $\tilde\phi_j\to \phi_j$ we can assume
\begin{equation}
\label{eq:form:zantei}
\begin{split}
p=-(\xi_0+\phi_1)(\xi_0-\phi_1)+\theta \phi_1^2+\Sigma_{j=2}^r\phi_j^2+\Sigma_{j=r+1}^d\phi_j^2,\\
\{\xi_0-\phi_1,\phi_j\}\sieq 0,\quad 0\leq j\leq r,\quad \{\phi_1,\phi_j\}\sieq 0,\quad 3\leq j\leq r,\\
 \{\phi_1,\phi_2\}(\bro)\neq 0,\quad {\rm det}(\{\phi_i,\phi_j\})_{1\leq i,j\leq r}\neq 0.
\end{split}
\end{equation}
Write
\[
A_1=(\{\phi_i,\phi_j\})_{1\leq i, j\leq r}\sieq \begin{pmatrix}0&{^t}a^{(1)}\\
-a^{(1)}&A_2\end{pmatrix},\quad a^{(1)}={^t}(\{\phi_1,\phi_2\},0,\ldots,0)
\]
and show that ${\rm dim Ker}A_2=1$. Since $A_2$ is skew symmetric of odd order $r-1$ then ${\rm dim Ker} A_2\geq 1$ 
hence ${\rm rank\,}A_2\leq r-2$. Suppose ${\rm rank\,}A_2< r-2$ then, expanding ${\rm det}A_1$ by cofactors of the first column (or the first row),  it is clear that
\[
{\rm det}A_1=0
\]
hence contradiction, thus ${\rm rank\,}A_2=r-2$.
This proves ${\rm dim Ker} A_2= 1$ and ${\rm Ker} A_2$ is spanned by $\be={^t}(\be_2,\ldots,\be_r)\neq 0$. Note that $\be_2\neq0$. Otherwise $\tilde\be={^t}(0,0,\be_3,\ldots,\be_r)\neq 0$ verifies $A_1\tilde\be=0$ which is a contradiction. Choosing a smooth orthogonal matrix $T=(t_{ij})_{2\leq i,j\leq r}$ with the first row $
\be/|\be|$ 
consider ${\tilde \phi_j}=\sum_{k=2}^rt_{jk}\phi_k$ so that $\sum_{j=2}^r{\tilde\phi_j^2}=\sum_{j=2}^r\phi_j^2$. Then we have
\[
\{\tilde\phi_2,\tilde\phi_j\}\sieq 0,\quad 2\leq j\leq r.
\]
We also note that
\[
\{\phi_1,\tilde\phi_2\}(\bro)=\Sigma_{k=2}^r(\be_k/|\be|)\{\phi_1,\phi_k\}(\bro)=(\be_2/|\be|)\{\phi_1,\phi_2\}(\bro)\neq 0.
\]
Since $(\{\tilde\phi_i,\tilde\phi_j\})_{2\leq i,j\leq r}=TA_2T^{-1}$ then 
\[
{\rm dim Ker}(\{\tilde\phi_i,\tilde\phi_j\})_{2\leq i,j\leq r}=1.
\]
Writing $\tilde\phi_j\to\phi_j$ ($2\leq j\leq r$) one can write
\begin{equation}
\label{eq:form:2bis}
\begin{split}
p=-(\xi_0+\phi_1)(\xi_0-\phi_1)+\theta \phi_1^2+\Sigma_{j=2}^r\phi_j^2+\Sigma_{j=r+1}^d\phi_j^2,\\
\{\xi_0-\phi_1,\phi_j\}
\sieq 0,\;\;0\leq j\leq d,\\
 \{\phi_2,\phi_j\}\sieq 0, \quad 2\leq j\leq r,\quad \{\phi_1,\phi_2\}(\bro)\neq 0.
\end{split}
\end{equation}
Since
\[
A_2=(\{\phi_i,\phi_j\})_{2\leq i, j\leq r}\sieq \begin{pmatrix}0&0\\
0&A_3\end{pmatrix}
\]
it is clear that $A_3=(\{\phi_i,\phi_j\})_{3\leq i,j\leq r}$ is nonsingular because ${\rm dim Ker}A_2=1$. 
Here note the following 
\begin{lem}{\rm(cf.\cite{MR2942985})}
\label{lem:BPP}Let $p$ have the form \eqref{eq:form:0} with \eqref{eq:form:1} and \eqref{eq:form:2}.
Then
\begin{align*}
\theta(\rho)=0&\Longleftrightarrow W(\rho)\neq \{0\},\\
\theta(\rho)>0& \Longleftrightarrow \text{$p$ is non-effectively hyperbolic at $\rho$ and $W(\rho)=\{0\}$},\\
(-1<)\;\;\theta(\rho)<0 &\Longleftrightarrow  \text{$p$ is effectively hyperbolic at $\rho$}.
\end{align*}
\end{lem}
\begin{proof}
Write $\delta\phi_1\to \phi_1$ with $\delta=(1+\theta)^{1/2}$  so that $p=-\xi_0^2+\Sigma_{j=1}^d\phi_j^2$.  If $F_pX=\mu X$ with $\mu\neq 0$ then $X\in {\rm Im\,}F_p$ and hence $X=\Sigma_{j=0}^d\eta_jH_{\phi_j}$. Let $\ep_0=-1$ and $\ep_j=1$ for $1\leq j\leq d$ and note that
\begin{equation}
\label{eq:koyuti}
\begin{split}
F_p (\sum_{j=0}^d\eta_jH_{\phi_j})=\sum_{i,j=0}^d\ep_i\{\phi_j,\phi_i\}\eta_jH_{\phi_i}=\sum_{i=0}^d\zeta_iH_{\phi_i},\\
 \zeta=A\eta,\quad A=(\ep_i\{\phi_j,\phi_i\})_{0\leq i,j\leq d}.
\end{split}
\end{equation}
Consider
\[
A=A'\oplus O_{d-r},\quad A'=\begin{pmatrix}0&{^t} a\\
 a&M
\end{pmatrix}.
\]
Note that $M$ is nonsingular provided $\delta\neq 0$.
From \eqref{eq:koyuti} a non-zero eigenvalue of $F_p(\rho)$ is also an eigenvalue of $A'$. 
Let $ a=M\al$. Since $M$ is skew-symmetric we see that
\begin{equation}
\label{eq:M:a}
A'\begin{pmatrix}1\\
-{\al}\end{pmatrix}=0,\quad M\al=a={^t}(\{\xi_0,\phi_1\},\ldots,\{\xi_0,\phi_r\}).
\end{equation}
Note that
\[
\begin{pmatrix}1&{^t}\al\\
0&1_d
\end{pmatrix}
\begin{pmatrix}\mu&{^t}(M\al)\\
M\al&\mu+M
\end{pmatrix}=\begin{pmatrix}\mu&\mu{^t}\al\\
M\al&\mu+M
\end{pmatrix}
\]
for ${^t}M=-M$ and that
\[
\begin{pmatrix}\mu&\mu{^t}\al\\
M\al&\mu+M
\end{pmatrix}\begin{pmatrix}1&0\\
-\al&1_d
\end{pmatrix}=\begin{pmatrix}\mu-\mu|\al|^2&\mu{^t}\al\\
-\mu\al&\mu+M
\end{pmatrix}.
\]
 Hence
\[
{\rm det}(\mu+A')={\rm det}\begin{pmatrix}\mu-\mu|\al|^2&\mu{^t}\al\\
-\mu\al&\mu+M
\end{pmatrix}=\mu\,{\rm det}\begin{pmatrix}1-|\al|^2&\mu{^t}\al\\
-\al&\mu+M
\end{pmatrix}.
\]
Therefore we have
\begin{equation}
\label{eq:det:A}
{\rm det}(\mu+A')=\mu G(\mu),\quad G(0)=(1-|\al|^2){\rm det}M
\end{equation}
where  ${\rm det}M>0$ because $M$ is non-singular skew-symmetric. 
Assume $W(\rho)=\{0\}$. Thanks to \cite[Lemma 1.4.4 ]{MR492751}, non-zero eigenvalues of $F_p(\rho)$ are semisimple (algebraic multiplicity=geometric multiplicity), denoted  $\lambda_1,\ldots,\lambda_r$ with (counting multiplicity) linearly independent eigenvectors $v_i=\Sigma_{j=0}^d\eta_{ij}H_{\phi_j}\in 
{\rm Im}F_p$, these  $\lambda_1,\ldots,\lambda_r$ are also semisimple eigenvalues of 
$A'$ with linearly independent eigenvectors $\eta_i={^t}(\eta_{i0},\ldots,\eta_{id})$. Since $G(\mu)$ is a monic polynomial of order $r$ we must have $G(\mu)=\Pi_{j=1}^r(\mu-\lambda_j)$. In particular this proves $|\al(\rho)|\neq 1$ and 
\begin{equation}
\label{eq:kankei}
(1-|\al(\rho)|^2){\rm det}\,M=\Pi_{j=1}^r \lambda_j(\rho),\quad \rho\in \Sigma,\quad W(\rho)=\{0\}.
\end{equation}
Assume $W(\rho)\neq \{0\}$ then $F_p(\rho)$ has $r-2$ nonzero semi-simple eigenvalues (with counting multiplicity). Suppose $|\al(\rho)|\neq 1$ so that $A'$ has $r$ nonzero eigenvalues.
It is clear that not all are semi-simple. So assume  $\lambda_j\neq 0$ is not a semi-simple eigenvalue of $A'$ so that there is $0\neq \eta={^t}(\eta_0,\ldots, \eta_r)$ such that $(A'-\lambda_j)\eta\neq 0$ and $(A'-\lambda_j)^k\eta=0$ with some $k\geq 2$. With $v=\Sigma_{j=0}^r\eta_jH_{\phi_j}(\rho)$ we have $(F_p(\rho)-\lambda_j)v\neq 0$ while $(F_p(\rho)-\lambda_j)^kv= 0$ which shows that $\lambda_j$ is nonzero and non semi-simple eigenvalue of $F_p$. This contradiction proves that $|\al(\rho)|=1$. Thus
\begin{equation}
\label{eq:kankei:2}
|\al(\rho)|=1\Longleftrightarrow W(\rho)\neq \{0\}.
\end{equation}
Thanks to \cite[Theorem 1.4.6]{MR492751}  we have either $F_p(\rho)=\{0,\pm i\mu_j(\rho), 1\leq j\leq r/2\}$ or $F_p(\rho)=\{0,\pm e(\rho), \pm i\mu_j(\rho), 1\leq j\leq r/2-1\}$ if $W(\rho)=\{0\}$ where $e(\rho)>0, \mu_j(\rho)>0$. The former case we have
\[
\Pi_{j=1}^r\lambda_j=\Pi_{j=1}^{r/2}(i\mu_j)(-i\mu_j)=\Pi_{j=1}^{r/2}\mu_j^2>0
\]
hence $|\al(\rho)|<1$ and vise versa. In the latter case we have
\[
\Pi_{j=1}^r\lambda_j=\Pi_{j=1}^{r/2-1}e(-e)(i\mu_j)(-i\mu_j)=-e^2\Pi_{j=1}^{r/2-1}\mu_j^2<0
\]
hence $|\al(\rho)|>1$ and vice versa.

Thus we have 
\begin{gather*}
|\al(\rho)|<1\Longleftrightarrow \text{$p$ is non-effectively hyperbolic at $\rho$ and $W(\rho)=\{0\}$},\\
|\al(\rho)|>1\Longleftrightarrow  \text{$p$ is effectively hyperbolic at $\rho$}.
\end{gather*}
 In view of \eqref{eq:form:1} one has 
\[
(\{\phi_i,\phi_j\})_{0\leq i,j\leq r}v\sieq 0,\quad v={^t}(1,-1/\delta,0,\ldots,0).
\]
On the other hand since ${\rm dim Ker\,}A'=1$ it follows from \eqref{eq:M:a} that $
\al={^t}(-1/\delta,0,\ldots,0)$ so that $\delta|\al|=1$, that is
\[
1+\theta=1/|\al|^2>0.
\]
This proves the assertion.

Thus applying this lemma one can prove \eqref{eq:te:seifu}.  
From $\theta=(1-|\alpha|^2)/|\alpha|^2$ it follows from \eqref{eq:kankei} that
\begin{equation}
\label{eq:kankei:b}
\Pi_{j=1}^r\lambda_j(\rho)=(|\al(\rho)|^2{\rm det}M(\rho))\theta(\rho),\quad \rho\in\Sigma',\;\;W(\rho)=\{0\}
\end{equation}
where $|\al(\rho)|^2{\rm det}M(\rho)$ is positive and smooth near $\bro$. 
\end{proof}
\begin{remark}\rm
Let $q=\Sigma_{j=1}^r\phi_j^2$ where $d\phi_j$ are linearly independent at $\bro$. Then
\begin{equation}
\label{eq:sei:tr}
{\rm Tr^{+}}A(\rho)={\rm Tr^{+}}F_p(\rho),\quad A=(\{\phi_i,\phi_j\})_{1\leq i,j\leq r}.
\end{equation}
\end{remark}
%
%
Let us apply Proposition \ref{pro:form:1} to a concrete example with $d=3$.  Consider
\begin{equation}
\label{eq:rei:2}
p=-\xi_0^2+\xi_1^2+(x_0+x_1-x_0x_2^k/k)^2\xi_n^2+\xi_2^2,\quad |x_i|\ll 1
\end{equation}
near $(0,e_n)$ with $n\geq 3$, $k\in\N$. It is clear that the doubly characteristic set is given by $
\Sigma=\{\xi_0=\xi_1=\xi_2=0, x_0+x_1-x_0x_2^{k}/k=0\}$. Rewrite $p$ in a normal form. Denoting 
\begin{gather*}
\varphi_1=(1-x_2^k/k)\rho^2(\xi_1-x_0x_2^{k-1}\xi_2),\quad \varphi_2=(x_0+x_1-x_0x_2^{k}/k)\xi_n,\\
 \varphi_3=\rho (x_0x_2^{k-1}\xi_1+\xi_2),\quad \rho=1/\sqrt{1+x_0^2x_2^{2(k-1)}}
\end{gather*}
one can write $
p=-(\xi_0-\varphi_1)(\xi_0+\varphi_1)+\theta\varphi_1^2+\varphi_2^2+\varphi_3^2$ 
where $\theta$ is given by $(1-x_2^k/k)^2(1+\theta)=1+x_0^2x_2^{2(k-1)}$, that is
\[
\theta=\big(2x_2^k/k+x_0^2x_2^{2(k-1)}-x_2^{2k}/k^2\big)/(1-x_2^k/k)^2
\]
 which is of normal form, indeed 
\[
\{\xi_0-\varphi_1,\varphi_2\}\equiv 0,\;\{\xi_0-\varphi_1,\varphi_j\}\sieq 0,\; j=0,1,3,\;\{\varphi_3,\varphi_j\}\sieq 0,\; j=0,1,2.
\]
Applying Proposition \ref{pro:form:1} we see that when $k\geq 1$ is odd $p$ is effectively hyperbolic in $\Sigma\cap\{x_2<g_k(x_0)\}$ and noneffectively hyperbolic of type 2 and of type 1 in $\Sigma\cap\{x_2=g_k(x_0)\}$ and $\Sigma\cap\{x_2>g_k(x_0)\}$ respectively where $g_1=1-\sqrt{1+x_0^2}$ and $g_k=0$ for $k\geq 3$. When $k\geq 2$ is even $p$ is noneffectively hyperbolic of type $1$ in $\Sigma\cap\{x_2\neq 0\}$ and of type 2 on $\Sigma\cap\{x_2=0\}$.


\section{Extension lemma}
\label{exlem}
As mentioned in Remark \ref{rem:O:4}, if $\tilde\theta|_{\Sigma'}=\theta|_{\Sigma'}$ we can write
\[
p=-(\xi_0+\tilde\phi_1)(\xi_0-\tilde\phi_1)+(\tilde\theta+O^2(\Sigma'))\tilde\phi_1^2+\Sigma_{j=2}^r\phi_j^2+\Sigma_{j=r+1}^d\phi_j^2
\]
which is still a normal form. Therefore the extension of $\theta|_{\Sigma'}$ to a full neighborhood of $\bro$ becomes an important issue.

\subsection{Extension of $C^{\infty}(\Sigma')$}
\begin{lem}
\label{lem:kakutyo}Assume that $\Sigma$ is given by $\phi_j=0$ $(0\leq j\leq d)$ satisfying \eqref{eq:form:1} and \eqref{eq:form:2}. There are neighborhoods $V_i$ of $\bro$  in $\R^{
n+1}\times (\R^n\setminus 0)$ verifying the followings: for any $\theta(x,\xi')\in C^{\infty}(\Sigma'\cap V_1)$ there exists an extension $\tilde\theta\in C^{\infty}(V_2)$ such that 
\begin{equation}
\label{eq:ext:1}
\{\phi_1, \tilde\theta\}=c_1\phi_1+c_2\phi_2,\quad \{\phi_2,\tilde\theta\}=c_3\phi_2\quad \text{in}\;\;V_2
\end{equation}
with smooth $c_{i}$ and
\begin{equation}
\label{eq:ext:2}
\{\phi_j,\tilde\theta\}\sieq 0,\quad 3\leq j\leq r.
\end{equation}
 Moreover we can take $\inf_{\Sigma'\cap V_1} \theta\leq {\tilde\theta}\leq \sup_{\Sigma'\cap V_1}\theta$ in $V_2$. 
\end{lem}
\begin{proof} 
Note that \eqref{eq:form:1} and \eqref{eq:form:2} imply $
\{\xi_0,\phi_2\}(\bro)\neq 0$ hence $\dif_{x_0}\phi_2(\bro)\neq 0$ then one can write
\[
\phi_2=e_2(x_0-\psi(x',\xi')),\quad e_2\neq 0,\quad x'=(x_1,\ldots,x_n),\;\;\xi'=(\xi_1,\ldots,\xi_n).
\]
In view of \eqref{eq:form:2} it follows that $d\psi(\bro)\neq 0$ because $\phi_1$ is independent of $\xi_0$. Therefore take $\Xi_0=\xi_0$, $X_0=x_0$ and $X_1=\psi(x',\xi')$ which satisfy the commutation relations and $d\Xi_0$, $dX_0$, $dX_1$, $\Sigma_{j=0}^n\xi_jdx_j$ are linearly independent at
$\bro$ hence extend to a full homogeneous symplectic coordinates system $(X,\Xi)$ (\cite[Theorem 21.1.9]{MR781536}). Changing $(X,\Xi)\to (x,\xi)$ one can assume that
\[
\phi_2=e_2(x_0-x_1).
\]
Since \eqref{eq:form:2} implies that $\dif_{\xi_1}\phi_1(\bro)\neq 0$ one can write
\[
\phi_1=e_1(\xi_1-\psi(x,\xi'')),\quad e_1\neq0,\quad \xi''=(\xi_2,\ldots,\xi_n).
\]
Now investigate $\Sigma'$ which is given by $\Sigma'=\{\phi_j=0,j=1,\ldots, d\}$ which is also given by
\begin{gather*}
\{\phi_1=\phi_2=0, \phi_j+\be_j\phi_2=0, 3\leq j\leq r, \phi_{r+1}=\cdots=\phi_d=0\}
\end{gather*}
where smooth $\be_j$ are free. Choosing $\be_k=\{\phi_1,\phi_k\}/\{\phi_2,\phi_1\}$ we have
\begin{equation}
\label{eq:xi:0}
\{\phi_j, \phi_k+\be_k\phi_2\}\sieq 0,\quad 3\leq k\leq r,\;\;j=0,1,2.
\end{equation}
In fact, if $j=2$ this is clear from \eqref{eq:form:2}. If $j=1$ it is clear from the choice of $\be_k$. The case $j=0$ is reduced to the case $j=1$ by \eqref{eq:form:1}. Writing $\phi_k+\be_k\phi_2\to\varphi_k$ ($3\leq k\leq r$) and  $\phi_j\to \varphi_j$ ($r+1\leq j\leq d$) the manifold $\Sigma'$ is given by
\begin{gather*}
\Sigma'=\{x_0-x_1=0, \xi_1-\tilde\psi(x',\xi'')=0, \tilde\varphi_j(x',\xi')=0, 3\leq j\leq d\}
\end{gather*}
where
\[
\tilde\psi(x',\xi'')=\psi(x_1,x_1,x'', \xi''),\quad \tilde\varphi_j(x',\xi')=\varphi_j(x_1,x_1,x'',\xi').
\]
Write  $\xi_1-\tilde\psi(x',\xi'')=e_1^{-1}\phi_1(x,\xi')+(x_0-x_1)f$ and $\tilde\varphi_j=\varphi_j+(x_0-x_1)f_j$. From $\{\xi_0,\varphi_j\}\sieq 0$ ($j\geq 3$) by \eqref{eq:xi:0} one has  $f_j\sieq 0$ ($j\geq 3$). Then we have
\begin{equation}
\label{eq:kankei:2-1}
\{\xi_1-\tilde\psi,\tilde\varphi_j\}\sieq 0,\quad 3\leq j\leq d
\end{equation}
for $\{\phi_1,\varphi_j\}\sieq 0$ and $\{x_0-x_1, \varphi_j\}\sieq 0$. It is also clear that
\[
\{\tilde\varphi_i,\tilde\varphi_j\}\sieq \{\varphi_i,\varphi_j\},\quad 3\leq i,j\leq d
\]
hence
\begin{equation}
\label{eq:kankei:3}
\{\tilde\varphi_i,\tilde\varphi_j\}\sieq 0,\quad 3\leq i\leq d, \;\;r+1\leq j\leq d.
\end{equation}
Set $\Xi_0=\xi_0$, $X_0=x_0$, $X_1=x_1$ and $\Xi_1=\xi_1-\tilde\psi(x',\xi'')$ which satisfy the commutation relations and $d\Xi_0$, $dX_0$, $dX_1$, $d\Xi_1$, $\Sigma_{j=0}^n\xi_jdx_j$ are linearly independent at $\bro$ hence extends to a homogeneous symplectic coordinates system $(X,\Xi)$ (\cite[Theorem 21.1.9]{MR781536}).  Since $\{\xi_0, \tilde\varphi_j\}=0$ and $\{x_0,\tilde\varphi_j\}=0$ writing $\tilde\Phi_j=\tilde\varphi_j$ for $3\leq j\leq d$ we have $\{\Xi_0,\tilde\Phi_j\}=0$ and $\{X_0,\tilde\Phi_j\}=0$ so that $\tilde\Phi_j=\tilde\Phi_j(X',\Xi')$. Now $\Sigma'$ is given by
\begin{gather*}
\Sigma'=\{X_0-X_1=0, \Xi_1=0, \hat\Phi_j(X',\Xi'')=0, 3\leq j\leq d\},\\
 \hat\Phi_j(X', \Xi'')=\tilde\Phi_j(X', 0,\Xi'').
\end{gather*}
Denote $
\Sigma''=\{\hat\Phi_j(X', \Xi'')=0, 3\leq j\leq d\}$ 
and show that $\Sigma''$ is cylindrical in the $X_1$ direction. Write
\begin{equation}
\label{eq:kankei:5}
\tilde\Phi_j(X',\Xi')=\hat\Phi_j(X', \Xi'')+\Xi_1f
\end{equation}
then we have
\[
\dif_{X_1}\hat\Phi_j(X',\Xi'')=\{\Xi_1, \hat\Phi_j\}=\{\Xi_1, \tilde\Phi_j-\Xi_1f\}\sieq 0
\]
because of \eqref{eq:kankei:2-1} and hence
\[
\dif_{X_1}\hat\Phi_j(X',\Xi'')\sieqb \Xi_1h_1+(X_0-X_1)h_2.
\]
Noting that the left-hand side contains neither $\Xi_1$ nor $X_0$ we conclude that
\[
\dif_{X_1}\hat\Phi_j(X',\Xi'')\sieqb 0,\quad 3\leq j\leq d
\]
which proves that $\Sigma''$ is cylindrical in the $X_1$ direction and hence
\[
\Sigma''=\{\hat\Phi_j(0,X'',\Xi'')=0, 3\leq j\leq d\}.
\]
Denote $\tilde\Sigma=\{\hat\Phi_j(0,X'',\Xi'')=0, 3\leq j\leq r\}$. 
Since the restriction of the symplectic form to $\tilde\Sigma$ has constant rank $r-2$ in a neighborhood of $\bro$. Indeed 
\[
{\rm det}(\{\tilde\varphi_i,\tilde\varphi_j\})_{3\leq i, j\leq r}(\bro)={\rm det}(\{\phi_i,\phi_j\})_{3\leq i,j\leq r}(\bro)\neq 0
\]
implies ${\rm rank}(\{\tilde\varphi_i,\tilde\varphi_j\})_{3\leq i, j\leq r}(\bro)=r-2$ on $\tilde\Sigma$. 
Thanks to \cite[Theorem 21.2.4]{MR781536},  there are homogeneous symplectic coordinates $X'', \Xi''$ such that, denoting $\hat\Phi_j(0,X'',\Xi'')$ by $\psi_j(X'',\Xi'')$, $3\leq j\leq d$ in these new symplectic cordinates $(X'',\Xi'')$, we have
\[
\tilde\Sigma=\{\psi_j=0, 3\leq j\leq r\}=\{X_2=\cdots=X_{l}=\Xi_2=\cdots=\Xi_{l}=0\},\quad r=2l
\]
so that $\Sigma'$ is given by
\begin{gather*}
\{X_0-X_1=0,\Xi_1=0,X_2=\cdots=X_{l}=\Xi_2=\cdots=\Xi_{l}=0, 
\psi_j(0, \tilde X,0,\tilde\Xi)\\
=0, r+1\leq j\leq d\},\quad \tilde X=(X_{l+1},\ldots,X_n),\;\tilde\Xi=(\Xi_{l+1},\ldots,\Xi_n).
\end{gather*}
Here we note that
\begin{equation}
\label{eq:kankei:add}
\{X_i, \psi_j\}\sieq 0,\quad \{\Xi_i,\psi_j\}\sieq 0,\quad 2\leq i\leq l,\;\;r+1\leq j\leq d.
\end{equation}
To prove this we first show
\begin{equation}
\label{eq:kankei:6}
\{\psi_i,\psi_j\}\sieq 0,\quad 3\leq i\leq d,\;\;r+1\leq j\leq d.
\end{equation}
 Note that $\dif_{X_1}\tilde\Phi_j(X',\Xi')\sieq 0$ for $3\leq j\leq d$ because $\hat\Phi_j(X',\Xi'')=0$ on $\Sigma''$ hence are linear combinations of $\hat\Phi_j(0,X'',\Xi'')$, $3\leq j\leq d$ and \eqref{eq:kankei:5}. Therefore we have
\begin{equation}
\label{eq:kankei:8}
\{\hat\Phi_i,\hat\Phi_j\}=\{\tilde\Phi_i+\Xi_1f_i,\tilde\Phi_j+\Xi_1f_j\}\sieq 0,\quad 3\leq i\leq d,\;\;r+1\leq j\leq d
\end{equation}
for $\{\tilde\Phi_i,\tilde\Phi_j\}\sieq 0$ by \eqref{eq:kankei:3}. Since 
\[
\{\hat\Phi_i(0,X'',\Xi''),\hat\Phi_j(0,X'',\Xi'')\}\sieq \{\hat\Phi_i+X_1g_i,\hat\Phi_j+X_1g_j\}\sieq \{\hat\Phi_i,\hat\Phi_j\}
\]
we get \eqref{eq:kankei:6}. Since $X_i$, $\Xi_i$, $2\leq i\leq l$ are linear combinations of $\psi_j(X'',\Xi'')$, $3\leq j\leq r$ we get the assertion.

Denote $\tilde\Sigma'=\{\tilde\psi_j(\tilde X,\tilde\Xi)=\psi_j(0,\tilde X,0,\tilde\Xi)=0, r+1\leq j\leq d\}$. Write
\[
\psi_j=\tilde\psi_j(\tilde X,\tilde\Xi)+\Sigma_{k=2}^lc_{jk}X_k+\Sigma_{k=2}^lc'_{jk}\Xi_k,\quad r+1\leq j\leq d.
\]
It follows from \eqref{eq:kankei:add} that $c_{jk}\sieq 0$ and $c'_{jk}\sieq 0$ hence we have
\[
\tilde\psi_j=\psi_j+O^2(\Sigma'),\quad r+1\leq j\leq d
\]
which proves that $\{\tilde\psi_i,\tilde\psi_j\}\sieq 0,\quad r+1\leq i,j\leq d$. Since $\tilde\psi_j$, $r+1\leq j\leq d$ contains no $X_0,X_1,\ldots,X_l$ and $\Xi_1,\ldots,\Xi_l$ we conclude that
\[
\{\tilde\psi_i,\tilde\psi_j\}{\overset{\tilde\Sigma'}{=}} 0,\quad r+1\leq i,j\leq d.
\]
Since the restriction of the symplectic form to $\tilde\Sigma'$ has rank $0$  there are homogeneous symplectic coordinates $\tilde X, \tilde\Xi$ such that $\tilde\Sigma'$ is given by (thanks to \cite[Theorem 21.2.4]{MR781536})
\[
\Xi_{l+1}=\cdots=\Xi_{d-l}=0.
\]
So there exist homogeneous symplectic coordinates $(X,\Xi)$ leaving $X_0, \Xi_0$ unchanged such that $\Sigma'$ is given by
\[
\{X_0-X_1=0,\Xi_1=0, X_2=\cdots=X_l=\Xi_2=\cdots=\Xi_l=0, \Xi_{l+1}=\cdots=\Xi_{d-l}=0\}.
\]
Let $\theta(x,\xi')\in C^{\infty}(\Sigma')$. Write $\theta(x,\xi')=\Theta(X,\Xi')$. Define the extension $\tilde\Theta(X,\Xi')$ of $\Theta(X,\Xi')$ outside $\Sigma'$ to a neighborhood of $\bar\varrho$ ($\bro\leftrightarrow \bar\varrho$) by
\begin{gather*}
\tilde\Theta(X,\Xi')=\Theta(X_0,X_0,0,\ldots,0, X_{l+1},\dots,X_n,0,\ldots,0,\Xi_{d-l+1},\ldots, \Xi_n).
\end{gather*}
It is clear that $\inf_{\Sigma'}\Theta\leq \tilde\Theta\leq \sup_{\Sigma'}\Theta$. Define the extension $\tilde\theta(x,\xi')$ of $\theta(x,\xi')$ outside $\Sigma'$ by
\[
\tilde\theta(x,\xi')=\tilde\Theta(X,\Xi').
\]
Since $\{\phi_2,\tilde\theta\}=\{\tilde e_2(X_0-X_1), \tilde \Theta\}=c_3(X_0-X_1)$ for $\{X_0-X_1,\tilde\Theta\}=\dif_{\Xi_1}\tilde\Theta=0$
thus we have
\[
\{\phi_2, \tilde\theta\}=c\,\phi_2.
\]
Note that $\phi_1$ is given by $\tilde e_1\Xi_1+\tilde f_1(X_0-X_1)$ then
\begin{gather*}
\{\phi_1,\tilde\theta\}=\{\tilde e_1\Xi_1+\tilde f_1(X_0-X_1), \tilde\Theta\}=\tilde c_1\Xi_1+\tilde c_2(X_0-X_1)=c_1\phi_1+c_2\phi_2
\end{gather*}
because $\{\Xi_1,\tilde\Theta\}=0$ which proves the assertion. 

Since $\{\Xi_j,\tilde\Theta\}=0$ ($1\leq j\leq l$) and $\{X_j,\tilde\Theta\}=0$ ($0\leq j\leq d-l$) and
\[
\hat\Phi_j(0,X'',\Xi'')=\Sigma_{k=2}^{l}a_{jk}\Xi_k+\Sigma_{k=2}^lb_{jk}X_k,\quad 3\leq j\leq r
\]
and $\hat\Phi_j=\hat\Phi_j(0,X'',\Xi'')+X_1g_j$ then noting $d-l\geq l$ we have $
\{\hat\Phi_j,\tilde\Theta\}\sieq 0$, $3\leq j\leq r$. 
Recalling $\tilde\Phi_j(X',\Xi')=\hat\Phi_j(X', \Xi'')+\Xi_1f_j$ we have $\{\tilde\Phi_j,\tilde\Theta\}\sieq 0$, $3\leq j\leq r$. Therefore we have
\[
\{\tilde\varphi_j,\tilde\theta\}\sieq 0,\quad 3\leq j\leq r.
\]
Since $\phi_j=\tilde\varphi_j+(x_0-x_1)g_j$, $3\leq j\leq r$ we conclude that $\{\phi_j,\tilde\theta\}\sieq 0$ for $3\leq j\leq r$. 
\end{proof}
Let us take a look back at  the example \eqref{eq:rei:1};
\[
p=-(\xi_0+\phi_1)(\xi_0-\phi_1)+\theta(x)\phi_1^2+\phi_2^2,\quad \phi_1=\xi_1,\;\;\phi_2=(x_0+x_1)\xi_n.
\]
Denoting
\[
\tilde\theta(x_0, x'')=\theta(x_0,-x_0,x''),\quad x''=(x_2,\ldots,x_n)
\]
one can write $\theta(x)=\tilde \theta(x_0,x'')+(x_0+x_1)\al(x)$. Applying Lemma \ref{lem:kakikae}
we can write with $\tilde\phi_1=(1+\nu)\phi_1$ 
\begin{gather*}
p=-(\xi_0+\tilde\phi_1)(\xi_0-\tilde\phi_1)+\hat\theta\tilde\phi_1^2+\phi_2^2,\quad \hat\theta=(\theta-\nu^2-2\nu)/(1+\nu)^2.
\end{gather*}
Choosing $2\nu=(x_0+x_1)\al(x)$ hence $
\hat\theta=\tilde\theta+r$, $r=O((x_0+x_1)^2)$  it is clear that 
\[
\{\tilde\phi_1,\hat\theta\}=O((x_0+x_1)),\quad \{\phi_2, \hat\theta\}= 0.
\]

Next, we reconsider the example \eqref{eq:rei:2}. Noting that $\dif/\dif x_1-x_0x_2^{k-1}\dif/\dif x_2$ is transversal to $x_0+x_1-x_0x_2^{k-1}=0$ we extend $\theta|_{\Sigma}$ to a neighborhood of $(0,0,0)$ along the characteristic curves such that $\{\varphi_1,\theta\}=0$. On the other hand $\{\varphi_2,\theta\}=0$ is obvious.

\subsection{More about the extension}

Denote by $\tilde\theta$ the extension of $\theta$ given by Lemma \ref{lem:kakutyo} then we see that
\begin{equation}
\label{eq:ext:seisitu}
\theta_1=\theta_2\Longrightarrow \tilde\theta_1=\tilde\theta_2,\quad \widetilde{\theta_1+\theta_2}=\tilde\theta_1+\tilde\theta_2,\quad \widetilde{\theta_1\theta_2}=\tilde\theta_1\tilde\theta_2.
\end{equation}
Let $\theta\in C^{\infty}(\Sigma')$. Taking that $H_{\xi_0-\phi_1}$ is tangent to $\Sigma'$ into account assume 
\begin{equation}
\label{eq:2:bibun}
\theta(\bro)=0,\quad \{\xi_0-\phi_1,\theta\}(\bro)=0,\quad \{\xi_0-\phi_1,\{\xi_0-\phi_1,\theta\}\}\neq0\;\; {\rm at}\;\;\bro.
\end{equation}
Denote by $\tilde \theta$ the extension of $\theta$.  Choosing a symplectic coordinates system such that $\Xi_0=\xi_0-\phi_1$, $X_0=x_0$ and denoting $\Theta(X,\Xi')=\tilde\theta(x,\xi')$ we have $\dif_{X_0}^k\Theta (\bar\varrho)=0$, $k=0,1$ and $\dif_{X_0}^2\Theta (\bar\varrho)\neq 0$ then thanks to Malgrange preparation theorem one can write
\[
\Theta(X,\Xi')=E\big((X_0-\Psi(X',\Xi'))^2+G(X',\Xi')\big)=E(F^2+G),\quad E\neq 0,\;G\geq 0
\]
where $\dif_{X_0}F=\{\Xi_0, F\}=1$, $\dif_{X_0}G=\{\Xi_0, G\}=0$ and $\dif_{X_0}\Psi=\{\Xi_0, \Psi\}=0$. Turning back to the coordinates $(x,\xi)$ we have
\begin{gather*}
\tilde\theta= e( f(x,\xi')^2+ g(x,\xi')),\quad f(x,\xi')=x_0-\psi(x,\xi'),\\
 \{\xi_0-\phi_1,  f\}=1,\;\;\{\xi_0-\phi_1, g\}=0.
\end{gather*}
Denoting by $\tilde e, \tilde f,\tilde g$ the extensions of  $e|_{\Sigma'}, f|_{\Sigma'}, g|_{\Sigma'}$ respectively we have by \eqref{eq:ext:seisitu} 
\[
{\tilde\theta}=\tilde e(\tilde f^2+\tilde g)
\]
where $\tilde e$, $\tilde f$ and $\tilde g$ verify \eqref{eq:ext:1} and \eqref{eq:ext:2} and that 
\[
\tilde f=x_0-\tilde\psi,\quad \dif_{x_0}\tilde\psi(\bro)=0,\quad \{\xi_0,\tilde g\}\sieq 0,\quad \tilde g\geq 0,\quad \tilde e\neq 0
\]
since $\{\xi_0-\phi_1,\tilde f\}=1$ and $\{\tilde f,\phi_1\}\sieq 0$.
%
%
%
\

\section{Tangent bicharacteristics}
\label{tanbich}

Let $p$ be of normal form 
\[
p=-(\xi_0+\phi_1)(\xi_0-\phi_1)+\theta\phi_1^2+\Sigma_{j=2}^r\phi_j^2+\Sigma_{j=r+1}^d\phi_j^2
\]
and recall that $\theta|_{\Sigma}$ is conformally invariant. 
Let $\tilde\theta$ be an extension of $\theta|_{\Sigma}$ then we say $d\theta$, $d\phi_j$ are linearly independent (resp. dependent) at $\bro$ if $d\tilde\theta$, $d\phi_j$ are linearly independent (resp. dependent) at $\bro$. This is independent of the choice of extensions of $\theta|_{\Sigma}$. Denote
\[
\nu=\{\xi_0-\phi_1,\{\xi_0-\phi_1,\theta\}\}(\bro),\quad \kappa=\{\{\xi_0-\phi_1,\phi_2\},\phi_2\}(\bro)/\{\phi_1,\phi_2\}(\bro).
\]
Then we have
\begin{prop}
\label{pro:tang_bicha:main}
Assume that $p$ is of normal form up to a term $O^4(\Sigma')$ with $\theta(\bro)=0$ and
\begin{equation}
\label{eq:r+1:j:cond} 
\{\xi_0-\phi_1,\theta\}(\bro)=0,\quad \{\phi_j,\theta\}(\bro)= 0,\;\; r+1\leq j\leq d.
\end{equation}
If $\kappa^2-4\nu> 0$ there exists a bicharacteristic $\ga$ of $p$ tangent to $\Sigma$ at $\bro$. More precisely, parametrizing $\ga$ by $x_0$ with $\ga(0)=\bro$ we have $\theta(\ga)=O(x_0^2)$ and $\phi_j(\ga)=O(x_0^2)$ for $j=0,\ldots, d$ where $\lim_{x_0\to 0}\phi_j(\ga)/x_0^{1+j}\neq 0$, $j=1,2$.
\end{prop}
\begin{cor}
\label{cor:tang_bicha:dep}
Assume that $p$ is of normal form up to a term $O^4(\Sigma')$ with $\theta(\bro)=0$. If $d\theta$, $d\phi_j$ are linearly dependent at $\bro$  and $\kappa^2-4\nu> 0$ there exists a bicharacteristic of $p$ tangent to $\Sigma$ at $\bro$.
\end{cor}
\begin{proof}If $d\theta$, $d\phi_j$ are linearly dependent at $\bro$ it is clear that \eqref{eq:r+1:j:cond} holds.
\end{proof}
\begin{cor}
\label{cor:te:teiti} Assume that $p$ is of normal form up to a term $O^4(\Sigma')$ with $\theta(\bro)=0$ and that $\theta|_{\Sigma}\geq 0$ or $\theta|_{\Sigma}\leq 0$ near $\bro$. If $\{\xi_0-\phi_1,\{\xi_0-\phi_1,\theta\}\}(\bro)= 0$ and $\{\{\xi_0-\phi_1,\phi_2\},\phi_2\}(\bro)\neq  0$ there exists a bicharacteristic of $p$ tangent to $\Sigma$ at $\bro$. 
\end{cor}
\begin{proof}By extension Lemma we may assume $\theta\geq 0$ or $\theta\leq 0$ in a neighborhood of $\bro$ hence we have $\nabla \theta(\bro)=0$ showing that $d\theta$, $d\phi_j$ are linearly dependent at $\bro$. Thus one can apply Corollary \ref{cor:tang_bicha:dep} with $\nu= 0$ and $\kappa\neq 0$. 
\end{proof}
\begin{cor}
\label{cor:te:futeiti} Assume that $p$ is of normal form up to a term $O^4(\Sigma')$ with $\theta(\bro)=0$ and $\theta|_{\Sigma}\leq 0$ near $\bro$. Then there exists a bicharacteristic of $p$ tangent to $\Sigma$ at $\bro$ unless $\{\xi_0-\phi_1,\{\xi_0-\phi_1,\theta\}\}(\bro)= 0$ and $\{\{\xi_0-\phi_1,\phi_2\},\phi_2\}(\bro)= 0$.
\end{cor}
\begin{proof}Note that \eqref{eq:r+1:j:cond} holds and $\nu\leq 0$.  If $\{\xi_0-\phi_1,\{\xi_0-\phi_1,\theta\}\}(\bro)\neq 0$ we have $\nu<0$ hence Proposition \ref{pro:tang_bicha:main} proves the assertion.
\end{proof}
Assume that $p$ is of normal form up to a term $O^4(\Sigma')$;
\begin{equation}
\label{eq:p:O4}
p=-(\xi_0+\phi_1)(\xi_0-\phi_1)+\theta\phi_1^2+\Sigma_{j=2}^r\phi_j^2+\Sigma_{j=r+1}^d\phi_j^2+O^4(\Sigma')
\end{equation}
With $\mu=\sqrt{1+\theta}>0$ one can rewrite
\[
p=-(\xi_0-\mu \phi_1)(\xi_0+\mu\phi_1)+\Sigma_{j=2}^r\phi_j^2+\Sigma_{j=r+1}^d\phi_j^2+O^4(\Sigma')
\]
where denoting $\tilde\phi_1=\mu \phi_1$ one sees
\begin{gather*}
\{\xi_0-\tilde\phi_1, \phi_j\}=\{\xi_0-\phi_1+(1-\mu)\phi_1,\phi_j\}\sieq (1-\mu)\{\phi_1,\phi_j\}\\
\sieq -\theta/\mu(1+\mu)\{\tilde\phi_1,\phi_j\}\sieq \hat\theta\{\tilde\phi_1,\phi_j\},\quad 0\leq j\leq d
\end{gather*}
where we have set 
\begin{equation}
\label{eq:te:te}
\hat\theta=-\theta/(\sqrt{1+\theta}+1+\theta).
\end{equation}
%
Writing $\mu\phi_1\to\phi_1$ we have
\[
p=-(\xi_0-\phi_1)(\xi_0+\phi_1)+\Sigma_{j=2}^r\phi_j^2+\Sigma_{j=r+1}^d\phi_j^2+O^4(\Sigma').
\]
To prove Proposition \ref{pro:tang_bicha:main} we follow \cite[Chapter 3]{MR3726883}.


\subsection{Case that $d\theta$, $d\phi_j$ are linearly independent}
\label{sec:indep}

In this section, we assume that $d\phi_j$, $d\theta$ are linearly independent at $\bro$. Hence $\Sigma\cap\{\theta=0\}$ is a submanifold of $\Sigma$ passing through $\bro$ and that $p$ is effectively hyperbolic in the side where $\theta<0$  and non-effectively hyperbolic type I on the other side where $\theta>0$. 
Taking \eqref{eq:te:te} into account we can assume
\begin{equation}
\label{eq:j:r:bisbis}
\{\phi_j, \hat\theta\}=O(\Sigma'),\quad 1\leq j\leq r
\end{equation}
thanks to the extension lemma. Recall that the assumption \eqref{eq:r+1:j:cond} implies
\begin{equation}
\label{eq:lin:dep}
\{\xi_0-\phi_1,\hat\theta\}(\bro)=0,\quad \{\phi_j,\hat\theta\}(\bro)=0,\;\; r+1\leq j\leq d.
\end{equation}
Choose a system of symplectic coordinates $(X,\Xi)$ such that $X_0=x_0$ and $\Xi_0=\xi_0-\phi_1$. Writing $(X,\Xi)\to (x,\xi)$ one has 
\[
p=-\xi_0^2-2\xi_0\phi_1+\Sigma_{j=2}^r\phi_j^2+\Sigma_{j=r+1}^d\phi_j^2+O^4(\Sigma')
\]
where
\begin{gather}
\label{eq:kityaku:1}
\{\phi_i,\phi_j\}\sieq 0,\; 0\leq i\leq d,\;\;j\geq r+1,\; \{\xi_0,\phi_j\}\sieq \hat\theta\,\{\phi_1,\phi_j\},\; 1\leq j\leq d,\\
\label{eq:kityaku:2}
\{\phi_1,\phi_2\}(\bro)\neq 0,\quad \{\phi_2,\phi_j\}\sieq 0, \;\; 3\leq j\leq r,\quad {\rm det}(\{\phi_i,\phi_j\})_{3\leq i,j\leq r}\neq 0
\end{gather}
and by \eqref{eq:j:r:bisbis}, \eqref{eq:lin:dep}
\begin{equation}
\label{eq:kityaku:3}
\begin{split}
 \dif_{x_0}^k\hat\theta(\bro)=0,\;\; k=0,1,\quad  \dif_{x_0}^2\hat\theta(\bro)=-\dif_{x_0}^2\theta(\bro)/2=-\nu/2,\\
  \{\phi_j,\hat\theta\}(\bro)=0,\quad 1\leq j\leq d.
  \end{split}
\end{equation}
\begin{lem}
\label{lem:te:r+1:a}We have
\[
\{\phi_2,\{\phi_j,\xi_0\}\}(\bro)=0,\quad r+1\leq j\leq  d.
\]
\end{lem}
\begin{proof}Note that $\{\xi_0,\{\phi_2,\phi_j\}\}(\bro)=0$ for $3\leq j\leq d$ by \eqref{eq:kityaku:1},   \eqref{eq:kityaku:2} and \eqref{eq:kityaku:3}. Since $
\{\phi_j, \{\xi_0,\phi_2\}\}(\bro)=\{\phi_j,\hat\theta\}(\bro)\{\phi_1,\phi_2\}(\bro)=0$ 
for $r+1\leq j\leq d$ by \eqref{eq:kityaku:1} and \eqref{eq:kityaku:3} then Jacobi's identity shows the assertion.
\end{proof}
From \eqref{eq:kityaku:1} we see that $d\hat\theta$, $dx_0, d\phi_j$, $0\leq j\leq d$ are linearly independent at $\bro$. Take 
\[
w=(\xi_0,x_0,\hat\theta, \phi_1,\ldots,\phi_d,\psi_1,\ldots,\psi_k)\quad (d+k=2n-1)
\]
to be a system of local coordinates around ${\bar\rho}$ so that $w(\bro)=0$. Note that we can assume that $\psi_j$ are independent of $x_0$ taking $\psi_j(0,x',\xi')$ as new $\psi_j$. 
Moreover we can assume that $\{\phi_i,\psi_j\}\sieq 0$ for $i=1,2$ taking $
\psi_j-\{\psi_j,\phi_2\}\phi_1/\{\phi_1,\phi_2\}-\{\psi_j,\phi_1\}\phi_2/\{\phi_2,\phi_1\}$ 
as new $\psi_j$. Thus it can be assumed that
\begin{equation}
\label{eq:psi:to:phi}
\{\xi_0,\psi_j\}\equiv 0,\quad \{\phi_i,\psi_j\}\sieq 0,\quad i=1,2,\;\;1\leq j\leq k.
\end{equation}
Let $\gamma(s)=(x(s),\xi(s))$ be a solution to the Hamilton equation 
\begin{equation}
\label{eq:Hamilton}
\frac{dx}{ds}=\frac{\dif p}{\dif \xi},\quad \frac{d\xi}{ds}=-\frac{\dif p}{\dif x}
\end{equation}
and recall $
df(\gamma(s))/ds=\{p,f\}(\gamma(s))$. 
We change the parameter from $s$ to $t$:
\[
t=s^{-1}
\]
so that we have $d/ds=-tD$ and $D=t(d/dt)$ and hence 
\[
\frac{d}{ds}(t^p F)=-t^{p+1}(DF+pF).
\]
 We now introduce new unknowns
\begin{equation}
\label{eq:okikae}
\begin{cases}
\xi_0(\gamma(s))=t^4\Xi_0(t),\;\;x_0(\gamma(s))=tX_0(t),\\
\hat\theta(\gamma(s))=t^2\Theta(t),\\
\phi_1(\gamma(s))=t^2\Phi_1(t),\;\;\phi_2(\gamma(s))=t^3\Phi_2(t),\\
\phi_j(\gamma(s))=t^4\Phi_j,\quad 3\leq j\leq r,\\
\phi_j(\gamma(s))=t^3\Phi_j(t),\quad r+1\leq j\leq d,\\
\psi_j(\gamma(s))=t^2\Psi_j(t),\quad 1\leq j\leq k
\end{cases}
\end{equation}
and write $
W=(\Xi_0,X_0, \Theta, \Phi_1,\ldots,\Phi_d,\Psi_1,\ldots,\Psi_k)$. 
In what follows $G(t, W)$, which may change from line to line,  denotes a smooth function in $(t, W)$ defined near $(0,0)$ such that $G(t,0)=0$. It is clear that 
\[
\{O^4(\Sigma'), \phi_j\}(\ga)=t^6G(t, W),\quad 0\leq j\leq d.
\]
Denote 
\begin{equation}
\label{eq:0toj}
 \{\xi_0, \phi_j\}=\hat\theta\{\phi_1,\phi_j\}+\sum_{i=1}^dC_i^j\phi_i, \;\;\kappa_j=C_1^j({\bar\rho}), \;\;\delta=\{\phi_1,\phi_2\}({\bar\rho})\neq 0
 \end{equation}
then from Lemma \ref{lem:te:r+1:a} we get
\begin{equation}
\label{eq:yakobi:2:bis}
\kappa_j=0,\quad j=r+1,\ldots,d.
\end{equation}
Here note that 
\begin{equation}
\label{eq:kazehiki}
\kappa_2=C_1^2({\bar\rho})=\{\{\xi_0,\phi_2\},\phi_2\}({\bar\rho)}/\{\phi_1,\phi_2\}({\bar\rho})
=\kappa.
\end{equation}
It is clear that
\begin{equation}
\label{eq:sesu:bb}
\{\xi_0,\phi_j\}(\ga)=\big(\delta_j\Theta+\kappa_j\Phi_1\big)t^2+t^3G(t, W),\quad \delta_j=\{\phi_1,\phi_j\}(\bro).
\end{equation}
Note that we have 
\[
\{\xi_0,\hat\theta\}(\ga)=-\nu X_0t/2+t^2G(t,W).
\]
Indeed write $\hat\theta(x,\xi')=\Sigma_{k=0}^2x_0^k\hat\theta_k(\tilde w)+O(x_0^3)$ where $w=(\xi_0,x_0,\tilde w)$ it follows from $\{\xi_0,\hat\theta\}(\bro)=0$ that $\hat\theta_1(\ga)=t^2G(t,W)$ which proves the assertion. 
The Hamilton equation is reduced to
\begin{equation}
\label{eq:modHamilton:indep}
\left\{\begin{array}{llllll}
D\Xi_0=-4\Xi_0+2\kappa_2\Phi_1\Phi_2+2\delta \Theta\Phi_2+tG(t,W),\\[3pt]
DX_0=-X_0+2\Phi_1+tG(t, W),\\[3pt]
D\Phi_1=-2\Phi_1+2\delta \Phi_2+tG(t, W),\\[3pt]
D\Theta=-2\Theta-\nu X_0\Phi_1+tG(t, W),\\[3pt]
D\Phi_2=-3\Phi_2+2\kappa_2\Phi_1^2+2\delta \Xi_0+2\delta\Phi_1\Theta+tG(t, W),\\[3pt]
tD\Phi_j=-4t\Phi_j+2\kappa_j\Phi_1^2+2\delta_j\Xi_0+2\delta_j\Phi_1\Theta\\[3pt]
\quad\quad \quad -2\Sigma_{k=3}^r\{\phi_k,\phi_j\}(\bro)\Phi_k+tG(t, W),\quad 3\leq j\leq r\\[3pt]
D\Phi_j=-3\Phi_j+tG(t, W),\;\;r+1\leq j\leq d,\\[3pt]
D\Psi_j=-2\Psi_j-2\sum_{k=r+1}^d\{\phi_k,\psi_j\}({\bar\rho})\Phi_k+tG(t, W),\;\;1\leq j\leq k.
\end{array}\right.
\end{equation}
We introduce a class of formal power series in $(t, \log{t})$ 
\[
{\mathcal E}=\{\sum_{0\leq j\leq i}t^i(\log t)^jw_{ij}\mid w_{ij}\in \C^N\}
\]
in which we look for a formal solution to the reduced Hamilton equation \eqref{eq:modHamilton:indep}.
\begin{lem}
\label{lem:formal}
There exists  $W=(\Xi_0,X_0,\Theta, \Phi_1, \ldots,\Phi_d,\Psi_1,\ldots,\Psi_k)\in {\mathcal E}$ such that 
$\Phi_1(0)\neq 0$ and $X_0(0)\neq 0$ satisfying \eqref{eq:modHamilton:indep}  formally.
\end{lem}
Assume that $W=(\Xi_0,X_0, \Theta, \Phi_1,\ldots,\Phi_d,\Psi_1,\ldots,\Psi_k)\in {\mathcal E}$ satisfies \eqref{eq:modHamilton:indep} formally. Denote
\begin{equation}
\label{eq:okuto}
\left\{\begin{array}{lllll}
X_0=\sum_{0\leq j\leq
i}t^i(\log{1/t})^jx^{0}_{ij},\;\;
\Xi_0=\sum_{0\leq
j\leq i}t^i(\log{1/t})^j\xi^{0}_{ij}\\
\Theta=\sum_{0\leq j\leq
i}t^i(\log{1/t})^j\theta_{ij},\\
\Phi_{\mu}=\sum_{0\leq j\leq
i}t^i(\log{1/t})^j\phi^{\mu}_{ij},\;\;\Psi_{\nu}=\sum_{0\leq j\leq
i}t^i(\log{1/t})^j\psi^{\nu}_{ij}
\end{array}\right.
\end{equation}
and $x^0_{00}=\bar x_0$, $\xi^0_{00}=\bar\xi_0$, $\theta_{00}=\bar\theta$, $\phi^{\mu}_{00}=\bar\phi_{\mu}$ and $\psi^{\nu}_{00}=\bar\psi_{\nu}$. Equating the constant terms of both sides of
\eqref{eq:modHamilton:indep} one has (except for the equations for $\Phi_j$ with $3\leq j\leq r$)
\begin{eqnarray*}
-4\bar\xi_{0}+2\kappa\bar\phi_{1}\bar\phi_{2}+2\delta \bar\theta\bar\phi_{2}=0,\quad-\bar x_{0}+2\bar\phi_{1}=0,
\quad
-2\bar\phi_{1}+2\delta \bar\phi_{2}=0,\\
-2\bar\theta-\nu \bar x_0\bar \phi_1=0,\quad -3\bar\phi_{2}+2\kappa
\bar\phi_{1}^2+2\delta\bar\xi_0+2\delta \bar\theta \bar\phi_{1}=0,\\
-3\bar\phi_j=0,\;\;r+1\leq j\leq d,\quad -2\bar\psi_j-2\Sigma_{k=r+1}^d\{\phi_k,\psi_j\}(\bro)\bar\phi_k=0.
\end{eqnarray*}
From the third line one has $\bar\phi_j=0$, $r+1\leq j\leq d$ and $\bar\psi_j=0$. 
Setting $b=\bar\phi_{1}$ we see 
\[
\bar x_{0}=2b,\quad\bar\phi_{2}=\delta^{-1}b,\quad \bar\theta=-\nu b^2,\quad 
2\bar\xi_{0}=\kappa \delta^{-1}b^2-\nu b^3
\]
 hence the second equation on the second line becomes
\begin{equation}
\label{eq:b:kimeru}
-3\delta^{-1}b+3\kappa b^2- 3\nu\delta  b^3=3b\big(-1/\delta+\kappa b -  \nu\delta  b^2\big)=0.
\end{equation}
Note that $\bar\phi_j$, $3\leq j\leq r$ are uniquely determined since ${\rm det}(\{\phi_i,\phi_j\}(\bro))_{3\leq i,j\leq r}\neq 0$. Let us study 
\begin{equation}
\label{eq:kosho:bis}
1/\delta-\kappa b + \nu \delta b^2=0.
\end{equation}
Since $\delta\neq 0$ we see that \eqref{eq:kosho:bis} has a nonzero real root $b=b(\kappa,\de)\neq 0$ if 
\[
\kappa^2-4\nu>0.
\]
Let us choose one of such $b\neq 0$ when $\nu<0$ and if $\nu>0$ we choose $b\neq 0$ such that $\delta\kappa b$ is smaller. Denote 
\[
\overline{W}=(\bar\xi_0,\bar x_0,\bar\theta, \bar\phi_j,\bar\psi_i)
\]
and look for a formal solution to \eqref{eq:modHamilton:indep} of the form $\overline{W}+W$ with $W\in {\mathcal E}^{\#}$ where 
\[
{\mathcal E}^{\#}=\{W=\sum_{1\leq i,0\leq j\leq i}t^i(\log t)^jw_{ij}\mid w_{ij}\in \C^{2n+2}\}.
\]
 To simplify notation we set
\[
\left\{\begin{array}{ll}W^I=(X_0,\Phi_2,\Xi_0,\Phi_1,\Theta),\quad W^{II}=(\Phi_{3},\ldots,\Phi_r)\\
W^{III}=(\Phi_{r+1},\ldots,\Phi_d),\quad W^{IV}=(\Psi_1,\ldots,\Psi_k)
\end{array}\right.
\]
then $W={^t}(W^I, W^{II}, W^{III}, W^{IV})=\sum_{1\leq i,0\leq j\leq i}t^i(\log t)^jw_{ij}$ satisfies
\begin{equation}
\label{eq:matform:a}
H\big(iw_{ij}-(j+1)w_{ij+1}\big)=Aw_{ij}+\delta_{i1}\delta_{j0}F+G_{ij}
\end{equation}
with $H=I\oplus O\oplus I\oplus I$ where $I$ and $O$ is the identity and zero matrix respectively and
\begin{equation}
\label{eq:A:form}
A=\left[\begin{array}{cccc}A_I&O&O&O\\
B_{II}&A_{II}&O&O\\
O&O&-3I&O\\
O&O&B_{III}&-2I
\end{array}\right].
\end{equation}
Moreover, $F$ is a constant vector and  
\[
G_{ij}=G_{ij}(w_{pq}\mid q\leq p\leq i-1),\quad G_{ij}=0,\;\;i=0,1.
\]
Here note that $|\la-A|=|\la-A_I||\la-A_{II}||\la+3I||\la+2I|$ and all eigenvalues of $A_{II}$ are pure imaginary. Making a more precise look on
$A_I$ we see 
\begin{equation}
\label{eq:A:I}
A_I=\left[\begin{array}{ccccc}
-1&0&0&2&0\\
0&-3&2\delta&2(\kappa b+\delta^{-1})&2\delta b\\
0&2\delta^{-1}&-4&2\kappa\delta^{-1}b&2b\\
0&2\delta&0&-2&0\\
-\nu b&0&0&-2\nu b&-2
\end{array}\right]
\end{equation}
where we have used \eqref{eq:kosho:bis}. To confirm that \eqref{eq:matform:a} can be solved successively we prove the following
\begin{lem}
\label{lem:koyuti} Assume $\kappa^2-4\nu> 0$. Then
$A_I$ has an eigenvalue $1$ and the other real eigenvalues are negative.
\end{lem}
\begin{proof} 
Expand $|\la-A_I|$ with respect to the last row we see
\begin{gather*}
|\la-A_I|=(\la+1)(\la+2)\left|\begin{array}{ccc}
\la+3&-2\delta&-2(\kappa b+\delta^{-1})\\
-2\delta^{-1}&\la+4&-2\kappa\delta^{-1}b\\
-2\delta&0&\la+2
\end{array}\right|\\
-2 \nu b (\la+2)\left|
\begin{array}{ccc}
\la+3&-2\delta&-2\delta b\\
-2\delta b&\la+4&-2b\\
-2\delta&0&0
\end{array}\right|\\
=(\la+1)(\la+2)(\la+6)(\la^2+3\la-4\kappa \delta b)+8\nu \delta^2 b^2(\la+2)(\la+6)\\
=(\la-1)(\la+2)(\la+6)\big(\la^2+5\la+8-4\kappa\delta b\big)
\end{gather*}
where we have used $\nu \delta^2 b^2=\kappa b\delta-1$.
Write
\[
\la^2+5\la+8-4\kappa\delta b=(\la+4)(\la+1)-4\nu\delta^2b^2.
\]
Then it is clear that real roots of $(\la+4)(\la+1)=4\nu\delta^2b^2$, if exist, are less than or equal to $-1$ if $\nu\leq 0$. Consider the case $\nu>0$ so that $\kappa^2-4\nu> 0$  is satisfied. Note that the roots  $b$ of \eqref{eq:kosho:bis} are given by
\[
b=\frac{\kappa}{2\nu\delta}\pm \frac{\sqrt{\kappa^2-4\nu}}{2\nu|\delta|}
\]
from which we have
\[
(\la+4)(\la+1)=4\nu\delta^2b^2=4\kappa\delta b-4=4+\frac{\sqrt{\kappa^2-4\nu}\,(\sqrt{\kappa^2-4\nu}\pm \delta \kappa/|\delta|)}{\nu/2}.
\]
By our choice of $b$ the right-hand side is less than $4$. This proves the assertion in the case $\nu>0$. 
\end{proof}
The rest of the proof of Proposition \ref{pro:tang_bicha:main} is just the repetition of the arguments in \cite[Sections 3.3 and 3.4]{MR3726883}. 
\begin{remark}
\label{rem:a}\rm Assume $\nu=\{\xi_0,\{\xi_0,\theta\}\}(\bro)=0$, $\kappa\neq 0$ and that
\[
\theta\geq \Sigma_{i=1}^lk_i^2\;\; \;\text{or}\;\;-\theta\geq \Sigma_{i=1}^lk_i^2
\]
where $d\phi_j$, $dk_i$ are linearly independent at $\bro$. Since $\{\xi_0,\{\xi_0,\pm\theta-\Sigma_{i=1}^lk_i^2\}\}(\bro)\geq 0$ it follows that $\{\xi_0,k_i\}(\bro)=0$, $1\leq i\leq l$. From this we see easily that
\begin{gather*}
dk_i(\ga(t))/dt=-\{p,k_i\}(\ga(t))/t^2=O(t),\quad t\to 0
\end{gather*}
which proves that $k_i(\ga(x_0))=O(x_0^2)$.
Therefore the bicharacteristic $\ga(x_0)$ is tangent to the manifold $\Sigma_0=\Sigma\cap\{k_i=0, i=1,\ldots,l\}\supset \Sigma\cap\{\theta=0\}$.
\end{remark}
\begin{remark}
\label{rem:b}\rm Assume $\theta|_{\Sigma}\geq 0$ or $\theta|_{\Sigma}\leq 0$ near $\bro$ and $\nu=\{\xi_0,\{\xi_0,\theta\}\}(\bro)\neq 0$. We can assume $\theta\geq 0$ or $\theta\leq 0$ near $\bro$. Thanks to  Malgrange preparation theorem one can write
\[
\theta=e\big((x_0-\psi(x',\xi'))^2+h(x',\xi')\big),\quad h(x',\xi')\geq 0,\quad e(\bro)\neq 0.
\]
With $f=x_0-\psi(x',\xi')$ it is clear that $\{\theta=0\}\subset\{f=0\}$. Note that $df$, $d\phi_j$ are linearly independent at $\bro$ because $\{\xi_0,f\}\neq 0$. Assume  $\kappa^2-4\nu>0$ hence there is a bicharacteristic $\ga$ tangent to $\Sigma$ by Proposition \ref{pro:tang_bicha:main}. Note that
\begin{gather*}
df(\ga(t))/dt=-\{p,f\}(\ga(t))/t^2\to 2\Phi_1(0)\{\xi_0, 
f\}(\bro),\quad t\to 0
\end{gather*}
which proves that 
\begin{gather*}
\lim_{x_0\to 0}df(\ga(x_0))/dx_0=\lim_{t\to 0}\big(df(\ga(t))/dt\big)\big(dt/dx_0\big)\\
=2\Phi_1(0)\{\xi_0, 
f\}(\bro)/X_0(0)=\{\xi_0, 
f\}(\bro)=1.
\end{gather*}
Thus the bicharacteristic $\ga(x_0)$ is transversal to the manifold $\Sigma\cap \{f=0\}$.
\end{remark}
%


\subsection{Case that $d\theta$, $d\phi_j$ are linearly dependent}


We first note that $d\theta$, $d\phi_j$ are linearly dependent at $\bro$ then
\begin{equation}
\label{eq:lin:dep-1}
\{\xi_0-\phi_1,\theta\}(\bro)=0,\quad \{\phi_j,\theta\}(\bro)=0,\;\; r+1\leq j\leq d.
\end{equation}
Repeat the same arguments as in Section \ref{sec:indep}. Choose a system of symplectic coordinates $(X,\Xi)$ such that $X_0=x_0$ and $\Xi_0=\xi_0-\phi_1$. Writing $(X,\Xi)\to (x,\xi)$ one has 
\[
p=-\xi_0^2-2\xi_0\phi_1+\Sigma_{j=2}^r\phi_j^2+\Sigma_{j=r+1}^d\phi_j^2+O^4(\Sigma')
\]
where \eqref{eq:kityaku:1}, \eqref{eq:kityaku:2} and \eqref{eq:kityaku:3} hold.
From \eqref{eq:kityaku:1} we see that $dx_0, d\phi_j$, $0\leq j\leq d$ are linearly independent at $\bro$. Take 
\[
w=(\xi_0,x_0,\phi_1,\ldots,\phi_d,\psi_1,\ldots,\psi_k)\quad (d+k=2n)
\]
to be a system of local coordinates around ${\bar\rho}$ so that $w(\bro)=0$. Recall that we can assume that \eqref{eq:psi:to:phi} holds. 
Let $\gamma(s)=(x(s),\xi(s))$ be a solution to the Hamilton equation as before. We make the same change of parameter from $s$ to $t=s^{-1}$ and introduce new unknowns \eqref{eq:okikae} where $\hat\theta$ is absent now 
%
%
and write $W=(\Xi_0,X_0,\Phi_1,\ldots,\Phi_d,\Psi_1,\ldots,\Psi_k)$ and $G(t, W)$ being as before.
\begin{lem}
\label{lem:te:Tay}
One can write
\begin{gather*}
\hat\theta(x,\xi')=\ell(\phi_2,\ldots, \phi_d)-\nu x_0^2/4+\theta_2(w')
+\theta_3(w')
\end{gather*}
where $\ell $ is a linear form in $(\phi_2,\ldots, \phi_d)$ and $\theta_2$ is a quadratic form in $w'=(x_0,\phi_1,\ldots, \phi_d, \psi_1,\ldots,\psi_k)$ containing no such term  $c\,x_0^2$ $(c\in\R)$ and $\theta_3(w')=O(|w'|^3)$.
\end{lem}
\begin{proof}
Since $\hat\theta(\bro)=0$ the Taylor formula gives $
\hat\theta(x,\xi')=\hat\theta_1(w')+\hat\theta_2(w')+O(|w'|^3)$ where $\hat\theta_1$ is a linear form in $(\phi_1,\ldots,\phi_d)$ for  $d\hat\theta$ and $d\phi_j$ are linearly dependent and $\hat\theta_2$ is a quadratic form in $w'$ . Since $\{\phi_2,\hat\theta\}(\bro)=0$ we see that $\hat\theta_1$ is independent of $\phi_1$. The rest of the proof is clear.
\end{proof}
Thanks to this lemma one has
\begin{equation}
\label{eq:te:hat}
\hat \theta(\ga)=-\nu X_0^2t^2/4+t^3G(t,W),\quad \hat\theta(\ga)=t^2G(t,W).
\end{equation}
Taking \eqref{eq:0toj} into account one has
\begin{equation}
\label{eq:sesu:b}
\{\xi_0,\phi_j\}(\ga)=(-\nu X_0^2\delta_j/4+\kappa_j\Phi_1)t^2+t^3G(t, W)
\end{equation}
where $\delta_j=\{\phi_1,\phi_j\}(\bro)$. Hence by \eqref{eq:kityaku:1} and \eqref{eq:yakobi:2:bis} we have
\begin{equation}
\label{eq:phi:xi:bis}
\{\phi_j,\xi_0\}(\ga)=t^3G(t, W),\quad r+1\leq j\leq d.
\end{equation}
Thanks to \eqref{eq:te:hat}, \eqref{eq:sesu:b} and \eqref{eq:phi:xi:bis}  the Hamilton equation is reduced to
\begin{equation}
\label{eq:modHamilton:dep}
\left\{\begin{array}{llllll}
D\Xi_0=-4\Xi_0+2\kappa\Phi_1\Phi_2-(\nu\delta/2) X^2_0\Phi_2+ tG(t,W),\\[3pt]
DX_0=-X_0+2\Phi_1+tG(t, W),\\[3pt]
D\Phi_1=-2\Phi_1+2\delta \Phi_2+tG(t, W),\\[3pt]
D\Phi_2=-3\Phi_2+2\kappa\Phi_1^2+2\delta \Xi_0-(\nu\delta/2) X_0^2 \Phi_1\\[3pt]
\quad\quad \quad+tG(t, W),\\[3pt]
tD\Phi_j=-4t\Phi_j+2\kappa_j\Phi_1^2+2\delta_j\Xi_0-(\nu\delta_j/2)X_0^2\Phi_1\\[3pt]
\quad\quad \quad -2\Sigma_{k=3}^r\{\phi_k,\phi_j\}(\bro)\Phi_k+tG(t, W),\quad 3\leq j\leq r\\[3pt]
D\Phi_j=-3\Phi_j+tG(t, W),\;\;r+1\leq j\leq d,\\[3pt]
D\Psi_j=-2\Psi_j-2\sum_{k=r+1}^d\{\phi_k,\psi_j\}({\bar\rho})\Phi_k+tG(t, W),\;\;1\leq j\leq k
\end{array}\right.
\end{equation}
(which is obtained from \eqref{eq:modHamilton:indep} by replacing $\Theta$ by $-\nu X_0^2/4$). We look for a formal solution to the reduced Hamilton equation \eqref{eq:modHamilton:dep}.
\begin{lem}
\label{lem:formal-1}
There exists a formal solution $W\in {\mathcal E}$ satisfying \eqref{eq:modHamilton:dep} with $\Phi_1(0)\neq 0$, $X_0(0)\neq 0$.
\end{lem}
Assume that $W=(\Xi_0,X_0,\Phi_1,\ldots,\Phi_d,\Psi_1,\ldots,\Psi_k)\in {\mathcal E}$ satisfies \eqref{eq:modHamilton:dep} formally. Denote $X_0$, $\Xi_0$, $\Phi_j$, $\Psi_j$ as \eqref{eq:okuto}
and $x^0_{00}=\bar x_0$, $\xi^0_{00}=\bar\xi_0$, $\phi^{\mu}_{00}=\bar\phi_{\mu}$ and $\psi^{\nu}_{00}=\bar\psi_{\nu}$. Equating the constant terms of both sides of
\eqref{eq:modHamilton:dep} one has 
\begin{eqnarray*}
-4\bar\xi_{0}+2\kappa\bar\phi_{1}\bar\phi_{2}-\nu\delta\bar x_{0}^2 \bar\phi_{2}/2=0,\quad-\bar x_{0}+2\bar\phi_{1}=0,\\
-2\bar\phi_{1}+2\delta \bar\phi_{2}=0,\;\;
-3\bar\phi_{2}+2\kappa
\bar\phi_{1}^2+2\delta\bar\xi_0-\nu\delta
\bar x_{0}^2\bar\phi_{1}/2=0,\\
-3\bar\phi_j=0,\;\;r+1\leq j\leq d,\quad -2\bar\psi_j-2\Sigma_{k=r+1}^d\{\phi_k,\psi_j\}(\bro)\bar\phi_k=0.
\end{eqnarray*}
From the third line one has $\bar\phi_j=0$, $r+1\leq j\leq d$ and $\bar\psi_j=0$. 
Setting $b=\bar\phi_{1}$ as before we see $\bar x_{0}=2b$, $\bar\phi_{2}=\delta^{-1}b$ and $
2\bar\xi_{0}=\kappa \delta^{-1}b^2-\nu b^3$ hence the second equation on the second line becomes
\[
-3\delta^{-1}b+3\kappa b^2- 3\nu\delta  b^3=3b\big(-1/\delta+\kappa b -  \nu\delta  b^2\big)=0
\]
which is the same as \eqref{eq:b:kimeru}. We choose the same $b\neq 0$ as in Section \ref{sec:indep}. 
Denote $
\overline{W}=(\bar\xi_0,\bar x_0,\bar\phi_j,\bar\psi_i)$ 
and look for a formal solution to \eqref{eq:modHamilton:dep} of the form $\overline{W}+W$ with $W\in {\mathcal E}^{\#}$. To simplify notation we set
\[
\left\{\begin{array}{ll}W^I=(X_0,\Phi_2,\Xi_0,\Phi_1),\quad W^{II}=(\Phi_{3},\ldots,\Phi_r)\\
W^{III}=(\Phi_{r+1},\ldots,\Phi_d),\quad W^{IV}=(\Phi_3,\ldots,\Phi_r)
\end{array}\right.
\]
then $W={^t}(W^I, W^{II}, W^{III}, W^{IV})=\sum_{1\leq i,0\leq j\leq i}t^i(\log t)^jw_{ij}$ satisfies \eqref{eq:matform:a}
with $A$ of the same form as \eqref{eq:A:form} where $A_I$ is replaced by
\begin{equation}
\label{eq:A:I:dep}
A_I=\left[\begin{array}{cccc}
-1&0&0&2\\
-2\nu\delta b^2&-3&2\delta&2(\kappa b+\delta^{-1})\\
-2\nu b^2&2\delta^{-1}&-4&2\kappa\delta^{-1}b\\
0&2\delta&0&-2\end{array}\right]
\end{equation}
where we have used \eqref{eq:kosho:bis}. To confirm that \eqref{eq:matform:a} can be solved successively we prove the following
\begin{lem}
\label{lem:koyuti:dep} Assume $\kappa^2-4\nu> 0$. Then
$A_I$ has an eigenvalue $1$ and the other real eigenvalues are negative.
\end{lem}
\begin{proof} Expanding $|\la-A_I|$ with respect to the first row we see
\begin{gather*}
|\lambda-A_I|=(\la+1)\left|
\begin{array}{ccc}
\la+3&-2\delta&-2(\kappa b+\delta^{-1})\\
-2\delta^{-1}&\la+4&-2\kappa \delta^{-1} b\\
-2\delta&0&\la+2
\end{array}\right|\\
+2\nu b \left|
\begin{array}{ccc}
2\delta b&\la+3&-2\delta\\
2b&-2\delta^{-1}&\la+4\\
0&-2\delta &0
\end{array}\right|\\
=(\la+1)(\la+6)(\la^2+3\la-4\kappa \delta b)+8\nu \delta^2 b^2(\la+6)\\
=(\la-1)(\la+6)\big(\la^2+5\la+8-4\kappa\delta b\big)
\end{gather*}
the rest of the proof is just a repetition of the proof of Lemma \ref{lem:koyuti}.
\end{proof}
Let us give a simple example. Consider
\begin{equation}
\label{eq:rei:3}
p=-\xi_0^2+(\xi_1+x_0\xi_n)^2+x_1^2(1+x_1^k+\nu(x))\xi_n^2,\quad |x_1|\ll1,\nu(0)=0\;\;
\end{equation}
near $(0, e_n)$, $e_n=(0,\ldots,0,1)$ with $ k\in\N$ and $n\geq 3$. When $k=1$ and $\nu(x)=0$  this $p$ was studied in \cite{MR2438425}, \cite{MR3726883} as a typical example where there is no spectral transition but a tangent bicharacteristic. It is clear that the doubly characteristic set is given by
\[
\Sigma=\{\xi_0=\xi_1+x_0\xi_n=0, x_1=0\}.
\]
Denoting 
\begin{gather*}
\varphi_1=-x_1(1+x_1^k)\xi_n/(1+(1+k/2)x_1^k)=-x_1\al(x_1)\xi_n,\quad \al(0)=1 
\end{gather*}
one can write
\[
p=-(\xi_0+\varphi_1)(\xi_0-\varphi_1)+(\theta+\nu/\al^2)\varphi_1^2+\varphi_2^2,\quad \varphi_2=\xi_1+x_0\xi_n
\]
where 
\begin{equation}
\label{eq:rei:3:b}
\theta=(1+k)x_1^k+k^2x_1^{2k}(1+x_1^k)^{-1}/4
\end{equation}
which is of normal form. Indeed we have
\[
\{\xi_0-\varphi_1, \varphi_2\}=-k(1+k)x^k_1\be(x_1)\xi_n,\;\; \{\xi_0-\varphi_1,\xi_0\}\equiv 0,\;\; \{\xi_0-\varphi_1,\varphi_1\}\equiv 0
\]
where $\be(0)=1/2$. The transition behavior is determined by $\nu|_{\Sigma}$ because $\theta|_{\Sigma}=0$. Note that
%
%
%
%
%
\[
\{\xi_0-\varphi_1,\theta+\nu/\al^2\}=\dif_{x_0}\nu/\al^2,\quad \{\xi_0-\varphi_1,\{\xi_0-\varphi_1,\theta+\nu/\al^2\}\}=\dif_{x_0}^2\nu/\al^2.
\]
Let $k=1$ then  $\{\{\xi_0-\varphi_1,\varphi_2\},\varphi_2\}/\{\varphi_1,\varphi_2\}=1$ at $(0,e_n)$ hence if $\dif_{x_0}\nu(0)=0$ and $\dif_{x_0}^2\nu(0)<1/4$  one can apply Proposition \ref{pro:tang_bicha:main} to conclude the existence of tangent bicharacteristic. 
For the case $k\geq 2$ some remarks will be given at the end of the next section.


\section{Elementary factorization}
\label{elfac}

\subsection{Elementary factorization}

Consider the principal symbol of \eqref{eq:1}
\begin{equation}
\label{eq:kojaku:ee}
p(x,\xi)=-\xi_0^2+A_{1}(x,\xi')\xi_0+A_{2}(x,\xi')
\end{equation}
where $A_j(x,\xi')\in S^j_{1,0}$ depending smoothly in $x_0$. We start with the following definition.
\begin{definition}
\label{dfn:one}\rm
 We say that $p(x,\xi)$ admits a local elementary factorization\index{local elementary factorization} if there exist real valued $\lambda(x,\xi')$, $\mu(x,\xi')\in S^1_{1,0}$ and $0\leq Q(x,\xi')\in S^2_{1,0}$ such that 
\[
p(x,\xi)=-\Lambda(x,\xi)M(x,\xi)+Q(x,\xi')
\]
with $\Lambda(x,\xi)=\xi_0-\lambda(x,\xi')$ and $M(x,\xi)=\xi_0-\mu(x,\xi')$ verifying with some $C>0$ 
\begin{eqnarray}
\label{eq:LaMQ}
&|\{\Lambda(x,\xi),Q(x,\xi')\}|\leq CQ(x,\xi'),\\
\label{eq:LAMQbis}
&|\{\Lambda(x,\xi),M(x,\xi)\}|\leq C \big(\sqrt{Q(x,\xi')}+|\Lambda(x,\xi')-M(x,\xi')|\big).
\end{eqnarray}
If we can find such symbols defined in a conic neighborhood of $\rho$ we say that $p(x,\xi)$ admits a microlocal elementary factorization at $\rho$. 
\end{definition}
Of course, elementary factorization is closely related to the classical derivation of energy estimates. To see this note that
\begin{gather*}
-\La M=-(\xi_0-\la+ic)\#(\xi_0-\mu-ic)-i\{\La, M\}/2-ic(\La-M)+S^0_{1,0}\\
=-\tilde\La\#\tilde M-i(\{\La, M\}/2-c(\La-M))+S^0_{1,0}
\end{gather*}
where $c\in S^0_{1,0}$ and $
\tilde\La=\xi_0-\la+ic$, $\tilde M=\xi_0-\mu-ic$. We also note 
\begin{gather*}
2{\mathsf{Im}}(\op{p}v, \op{\tilde\La}v)=\frac{d}{dx_0}\big(\|\op{\tilde\La}v\|^2+(\op{Q}v,v)\big)\\
-2{\mathsf{Re}}(\op{\tilde\La}v, \op{\{\La, M\}/2+c(\La-M)}v)
-{\mathsf{Re}}(\op{\{\tilde\La, Q\}}v, v)/2
\end{gather*}
modulo $C(\|\op{\tilde\La}v\|^2+\|v\|^2)$. Assuming that one can find $c\in S^0_{1,0}$ such that $|\{\La, M\}/2+c(\La-M)|\lesssim \sqrt{Q}$ and taking $|\{\tilde\La,Q\}-\{\La,Q\}|\lesssim\sqrt{Q}$ into account we will obtain energy estimates.
\begin{lem}[\cite{zbMATH03625907}] 
\label{lem:fact:bicha} If $p$ admits a microlocal elementary factorization\index{microlocal elementary factorization} at $\rho$ there is no bicharacteristic tangent to $\Sigma$ at $\rho$.
\end{lem}
In this section, we restrict ourselves to the case 
\[
{\rm Spec}F_p(\rho)\subset i\R\;\;\big(\Longleftrightarrow \Pi \la_j(\rho)\geq 0\Longleftrightarrow \theta|_{\Sigma}\geq 0\big)\quad\rho\in \Sigma.
\]
Applying the extension lemma to $\theta\big|_{\Sigma}$ one can assume from the beginning that
\[
p=-(\xi_0+\phi_1)(\xi_0-\phi_1)+\theta\phi_1^2+\Sigma_{j=2}^r\phi_j^2+\Sigma_{j=r+1}^d\phi_j^2+R
\]
where $R=O(|\phi|^4\xig^{-2})=O^4(\Sigma')$ with $\phi=(\phi_1,\ldots,\phi_d)$ and that
\begin{equation}
\label{eq:ph:to:te}
\theta\geq 0,\quad \{\phi_j,\theta\}\sieq 0,\quad j=1,\ldots,r.
\end{equation}
\begin{prop}
\label{pro:kihon:bunkai}Assume that $p$ is of normal form up to term $O^4(\Sigma')$ and satisfies $\theta\geq 0$ near $\bro$. If
\begin{gather*}
|\{\xi_0-\phi_1,\theta\}|\leq C(\sqrt{\theta}+|\phi_1|+\sqrt{|\phi'|})^2,\\
 |\{\{\xi_0-\phi_1,\phi_2\},\phi_2\}|\leq C(\sqrt{\theta}+|\phi_1|+\sqrt{|\phi'|}),\;\;\phi'=(\phi_2,\ldots,\phi_d)
\end{gather*}
holds then $p$ admits a microlocal elementary factorization at $\bro$.
\end{prop}
\begin{cor}
\label{cor:kihon:bunkai}Assume that $p$ is of normal form up to term $O^4(\Sigma')$ and satisfies $\theta\geq 0$ near $\bro$ and $\{\xi_0-\phi_1,\theta\}\sieq 0$. If $\{\{\xi_0-\phi_1,\phi_2\},\phi_2\}\sieq 0$ 
then $p$ admits a microlocal elementary factorization at $\bro$ while $p$ does not if $\{\{\xi_0-\phi_1,\phi_2\},\phi_2\}(\bro)\neq 0$.
\end{cor}
\begin{proof}From the assumption one can write $
\{\xi_0-\phi_1,\theta\}=\Sigma_{k=1}^dc_d\phi_k$. 
Note that $\{\xi_0-\phi_1,\{\phi_2,\theta\}\}\sieq 0$ and $|\{\theta,\{\xi_0-\phi_1,\phi_2\}\}|\lesssim \sqrt{\theta}$
 for $\theta\geq 0$. It follows from Jacobi's identity we have 
 \[
 |\{\phi_2,\{\xi_0-\phi_1,\theta\}\}|\lesssim |\phi|+\sqrt{\theta},\quad \phi=(\phi_1,\ldots,\phi_d).
 \]
 This shows that $|\{\xi_0-\phi_1,\theta\}|\lesssim (|\phi|+\sqrt{\theta})|\phi_1|+|\phi'|$ hence the first condition is satisfied. The second condition is obviously satisfied and the assertion follows from Proposition \ref{pro:kihon:bunkai}. Since $\{\xi_0-\phi_1,\{\xi_0-\phi_1,\theta\}\}\sieq 0$ for $\{\xi_0-\phi_1,\theta\}\sieq 0$ if $\{\{\xi_0-\phi_1,\phi_2\},\phi_2\}(\bro)\neq 0$ there is a bicharacteristic tangent to $\Sigma$ at $\bro$ by Corollary \ref{cor:te:teiti} hence Lemma \ref{lem:fact:bicha} proves the last assertion.
 \end{proof}
Let
\[
\La=\phi_1+\ell(\phi'')\phi_1-\la\phi_1^3\xig^{-2},\quad \ell(\phi'')=\Sigma_{j=3}^r\be_j\phi_j,\quad\phi''=(\phi_3,\ldots,\phi_r)
\]
where $\la>0$ is a parameter and $\be_j$ are smooth which are determined later. Write
\[
p=-(\xi_0+\La)(\xi_0-\La)+Q
\]
where $
Q=\theta\phi_1^2+\tilde Q$ 
with
\begin{gather*}
\tilde Q=\Sigma_{j=2}^d\phi_j^2-2\ell\phi_1^2(1+\ell/2)+2\la\phi_1^4\xig^{-2}(1+\ell-\la\phi_1^2\xig^{-2}/2)+R.
\end{gather*}
Taking $\la>0$ large (since we are working in a neighborhood or $\bro$ it can be assumed that $|\phi|\xig^{-1}$ is arbitrarily small) it is clear that there is $c>0$ such that
\[
 Q\geq c\,(|\phi'|^2+\theta\phi_1^2+\phi_1^4\xig^{-2}),\quad \phi'=(\phi_2,\ldots,\phi_d).
\]
Note that
\begin{gather*}
\{\xi_0-\La, \tilde Q\}=\{\xi_0-\phi_1, |\phi'|^2-2\ell(\phi'') \phi^2_1(1+\ell(\phi'')/2)\}\\
-\{\ell(\phi'')\phi_1,|\phi'|^2\}+O(Q).
\end{gather*}
Recall that one can write
\[
\{\xi_0-\phi_1,\phi_j\}=\Sigma_{k=1}^d\al_{jk}\phi_k,\quad 1\leq j\leq d.
\]
Moreover since $\{\phi_2,\{\xi_0-\phi_1,\phi_j\}\}\sieq 0$ for $r+1\leq j\leq d$ which follows from Jacobi's identity, one has
\[
\al_{j1}\sieq 0,\quad r+1\leq j\leq d.
\]
Therefore we have
\begin{gather*}
\{\xi_0-\La, \tilde Q\}=2\sum_{j=2}^d\phi_j\sum_{k=1}^d\al_{jk}\phi_k-2\phi_1^2\sum_{j=3}^r\be_j\sum_{k=1}^d\al_{jk}\phi_k(1+\ell(\phi'')/2)\\
-2\phi_1\sum_{j=2}^d\phi_j\sum_{k=3}^r\be_k\{\phi_k,\phi_j\}+O(Q).
\end{gather*}
Note that
\begin{gather*}
\phi_j\sum_{k=1}^d\al_{jk}\phi_k=O(Q),\;\;r+1\leq j\leq d,\quad
\phi_1^2\sum_{j=3}^r\be_j\sum_{k=1}^d\al_{jk}\phi_k\ell(\phi'')=O(Q),\\
\phi_1\phi_j\sum_{k=3}^r\be_k\{\phi_k,\phi_j\}=O(Q),\;\;r+1\leq j\leq d
\end{gather*}
we have
\begin{gather*}
\{\xi_0-\La, \tilde Q\}=2\sum_{j=2}^r\phi_j\sum_{k=1}^d\al_{jk}\phi_k-2\phi_1^2\sum_{j=3}^r\be_j\sum_{k=1}^d\al_{jk}\phi_k\\
-2\phi_1\sum_{j=2}^r\phi_j\sum_{k=3}^r\be_k\{\phi_k,\phi_j\}+O(Q)\\
=2\phi_1\sum_{j=2}^r\al_{j1}\phi_j-2\phi_1^3\sum_{j=3}^r\be_j\al_{j1}
-2\phi_1\sum_{j=3}^r\phi_j\sum_{k=3}^r\be_k\{\phi_k,\phi_j\}+O(Q)\\
=2\phi_1\sum_{j=3}^r\phi_j\Big(\al_{j1}-\sum_{k=3}^r\{\phi_k,\phi_j\}\be_k\Big)+2\al_{21}\phi_1\phi_2-2\phi_1^3\sum_{j=3}^r\be_j\al_{j1}+O(Q)
\end{gather*}
because $\{\phi_k,\phi_2\}\sieq 0$ for $k\geq 3$. Choose $\be_k$ such that
\[
\sum_{k=3}^r\{\phi_k,\phi_j\}\be_k=\al_{j1},\quad 3\leq j\leq r.
\]
Since $(\{\phi_k,\phi_j\})_{3\leq k, j\leq r}$ is skew symmetric and nonsingular we have
\[
\sum_{j=3}^r\be_j\al_{j1}=0.
\]
Therefore we conclude
\[
\{\xi_0-\La, \tilde Q\}=2\al_{21}\phi_1\phi_2+O(Q).
\]
Since $\al_{21}\sieq \{\{\xi_0-\phi_1,\phi_2\}, \phi_2\}/\{\phi_1,\phi_2\}$ we have
\[
|\{\{\xi_0-\phi_1,\phi_2\},\phi_2\}|\leq C(\sqrt{\theta}+|\phi_1|+\sqrt{|\phi'|})\Longrightarrow \{\xi_0-\La, \tilde Q\}=O(Q).
\]
It remains to study $\{\xi_0-\La, \theta\phi_1^2\}$. Taking \eqref{eq:ph:to:te} into account we have
\begin{gather*}
\{\xi_0-\La, \theta\phi_1^2\}=\{\xi_0-\phi_1, \theta\phi_1^2\}+O(Q)\\
=2\theta\phi_1\{\xi_0-\phi_1,\phi_1\}+\phi_1^2\{\xi_0-\phi_1,\theta\}+O(Q)
=\phi_1^2\{\xi_0-\phi_1,\theta\}+O(Q).
\end{gather*}
Therefore we conclude that
\[
|\{\xi_0-\phi_1,\theta\}|\leq C(\theta+\sqrt{\theta}|\phi_1|+|\phi'|)\Longrightarrow \{\xi_0-\La, \theta\phi_1^2\}=O(Q).
\]
Consider $\{\xi_0-\La,\xi_0+\La\}=2\{\xi_0,\La\}$ which is
\[
2\{\xi_0,\phi_1+\ell(\phi'')\phi_1-\la\phi_1^3\xig^{-2}\}\sieq 0
\]
for $\{\xi_0,\phi_1\}\sieq 0$. Since $|\phi'|\lesssim \sqrt{Q}$ and $|\phi_1|\lesssim |\La|$ we have
\[
|\{\xi_0-\La,\xi_0+\La\}|\lesssim \sqrt{Q}+|\La|
\]
which completes the proof.
\qed

\medskip

Reconsider the example \eqref{eq:rei:2} with $k\geq 2$. If $\theta=c\,x_2^l$ with $c>0$ and even $l$ it is easy to check that
\[
|\{\xi_0-\varphi_1,\theta\}|\lesssim \theta
\]
hence in view of Proposition \ref{pro:kihon:bunkai} $p$ admits a microlocal elementary factorization at $\bro$. 

Next, reconsider $p$ in \eqref{eq:rei:3} with $k\geq 2$ so that $\{\{\xi_0-\varphi_1,\varphi_2\},\varphi_2\}|_{\Sigma}=0$. If we take $\nu(x)= x_2^l$ with even $l$ it is clear $\{\xi_0-\varphi_1,\theta+\nu\}\equiv 0$ and $(\theta+\nu)|_{\Sigma}\geq 0$ hence it follows from Corollary \ref{cor:kihon:bunkai} that $p$ admits a microlocal elementary factorization at $\bro$. In particular, there is no tangent bicharacteristic. 


\bibliographystyle{plain}
\bibliography{bibTrans}

\end{document}